\documentclass[a4paper,11pt]{article}
\usepackage{stmaryrd}
\usepackage{bbm}
\usepackage[top=2.5cm, bottom=2.6cm, left=2.8cm, right=2.8cm]{geometry}
\usepackage{amsfonts}
\usepackage{mathrsfs}
\usepackage{amsmath}
\usepackage{latexsym}
\usepackage{color}
\usepackage{amscd}
\usepackage{amssymb}
\usepackage{amsthm}
\usepackage{latexsym}
\usepackage{indentfirst}
\newtheorem{theorem}{Theorem}[section]
\newtheorem{lemma}[theorem]{Lemma}
\newtheorem{proposition}[theorem]{Proposition}
\newtheorem{corollary}[theorem]{Corollary}

\theoremstyle{definition}
\newtheorem{definition}[theorem]{Definition}

\begin{document}

\title{On crystal bases of two-parameter $(v, t)$-quantum groups}
\author{Weideng Cui}
\date{}
\maketitle \begin{abstract}Following Kashiwara's algebraic approach in one-parameter case, we construct crystal bases for two-parameter quantum algebras and for their integrable modules. We also show that the global crystal basis coincides with the canonical basis geometrically constructed by Fan and Li up to a 2-cocycle deformation.\end{abstract}

 \large

\maketitle
\section{Introduction}

The theory of canonical basis for one-parameter quantum algebra was developed by Lusztig first in the ADE case (see [L1]) and subsequently in the general case (see [L3] and [L4]). In [Kas1] and [Kas2], Kashiwara constructed the crystal basis and the global crystal basis for the one-parameter quantum algebra associated to an arbitrary symmetrizable generalized Cartan matrix. The canonical basis and the global crystal basis of one-parameter quantum algebra were proved to be the same by Lusztig for type ADE in [L2] and by Grojnowski-Lusztig in the general case in [GL]. The canonical basis and the crystal basis have many remarkable properties, such as positivity, and promote the development in many areas of mathematics which are greatly related to quantum groups (see [Ar], [FZ] and [KL]).

Benkart, Kang and Kashiwara ([BKK]) constructed the crystal bases for quantum $\mathfrak{g}\mathfrak{l}(m|n)$ and for its polynomial representations at $q=0.$ Jeong, Kang and Kashiwara ([JKK]) constructed the global crystal bases for the quantum generalized Kac-Moody algebras and for their integrable modules. Clark, Hill and Wang ([CHW1]) constructed the crystal bases and its global version for quantum supergroups $\mathbf{U}_{q,\pi}(A)$ of anisotropic type and for their integrable modules. Recently, they ([CHW2]) constructed the canonical bases (=global crystal bases) for quantum supergroups when the Cartan data is of type $\mathfrak{g}\mathfrak{l}(m|1),$ $\mathfrak{o}\mathfrak{s}\mathfrak{p}(1|2n)$ and $\mathfrak{o}\mathfrak{s}\mathfrak{p}(2|2n),$ and also provided a new self-contained construction of canonical bases in the non-super case.

On the other hand, two-parameter and multi-parameter quantum algebras have been widely studied from the early 1990s by many authors; see [BGH1-2, BW1-2, C1-2, D1-5, FL1-2, HP1-2, PHR, JZ, R] and the references therein. Many facts and properties about the representation theory of one-parameter quantum algebras can be generalized to the two-parameter case, such as the semi-simplicity of the category of integrable modules.

Let $\mathfrak{g}$ be a symmetrizable Kac-Moody algebra. Fan and Li [FL1] constructed an algebra associated to $\mathfrak{g}$ from the mixed version of Lusztig's geometric framework by using mixed perverse sheaves on a quiver variety and Deligne's weight theory, and showed that this algebra is isomorphic to the negative part $U_{v,t}^{-}(\mathfrak{g})$ of a two-parameter quantum algebra $U_{v,t}(\mathfrak{g})$, in which the second parameter corresponds to the Tate twist.

From this geometric setting, they also obtained a basis for $U_{v,t}^{-}(\mathfrak{g})$, which is consisting of simple perverse sheaves of weight zero. If one forgets the Tate twist, this basis is exactly the same as the canonical basis in the one-parameter case. It also admits many favorable properties such as positivity and integrality as does its one-parameter analog. In a sequel paper [FL2], they showed that the categories of weight modules of this two-parameter quantum algebra and its one-parameter analog (letting $t=1$) are equivalent. Moreover, the integrable modules are preserved. Thus one obtains the semi-simplicity of the category of integrable $U_{v,t}(\mathfrak{g})$-modules.

In this paper, we develop the crystal basis theory for the two-parameter quantum algebras $U_{v,t}(\mathfrak{g})$ and for their integrable modules following the framework given in [Kas2]. We also show that the global crystal basis coincides with the canonical basis constructed in [FL1] up to a 2-cocycle deformation.

We organize this paper as follows. In Section 2, we study the two-parameter quantum algebra $U_{v,t}(\mathfrak{g})$ constructed in [FL1] and their integrable modules. In Section 3, we define the notion of crystal lattice and crystal basis for integrable modules in the two-parameter setting. We establish the tensor product rule of crystal bases and formulate a polarization on an integrable module. In Section 4, we introduce a two-parameter Kashiwara algebra in order to formulate the crystal basis of $U_{v,t}^{-}(\mathfrak{g})$, and we also establish its basic properties. Furthermore, we introduce a bilinear form (called polarization) on $U_{v,t}^{-}(\mathfrak{g})$ and the Kashiwara operators. In Section 5, we adapt Kashiwara's grand loop inductive argument to prove the existence theorem for crystal bases. In Section 6, we study further properties of the polarization, and show that the crystal basis is orthonormal with respect to the polarization at $v=0$. In Section 7, we prove the existence of global crystal bases for $U_{v,t}^{-}(\mathfrak{g})$ and all integrable modules $V(\lambda)$ with $\lambda\in P_{+}.$ We also show that the crystal bases and the global crystal bases are invariant under the star operator $*$.

\section{Two-parameter $(v,t)$-quantum algebras}

\subsection{Definition of $U_{v,t}(\mathfrak{g})$}

For any $n, k\in \mathbb{Z}_{\geq 0},$ and $n\geq k$, we set
$$[n]_{v}=\frac{v^{n}-v^{-n}}{v-v^{-1}},~~~[n]_{v}^{!}=\prod_{r=1}^{n}[r]_{v},~~~\left[\hspace*{-0.1cm}\begin{array}{c}n \\
k
\end{array}\hspace*{-0.1cm}\right]_{v}=\frac{[n]_{v}^{!}}{[k]_{v}^{!}[n-k]_{v}^{!}};$$
$$[n]_{v,t}=\frac{(vt)^{n}-(vt^{-1})^{-n}}{vt-(vt^{-1})^{-1}},~~~[n]_{v,t}^{!}=\prod_{r=1}^{n}[r]_{v,t},~~~\left[\hspace*{-0.1cm}\begin{array}{c}n \\
k
\end{array}\hspace*{-0.1cm}\right]_{v,t}=\frac{[n]_{v,t}^{!}}{[k]_{v,t}^{!}[n-k]_{v,t}^{!}}.$$

We shall review the definition of $U_{v,t}(\mathfrak{g})$ following [FL1]. Suppose that the following data are given.\vskip2mm
\hspace*{2.4em} (a) a finite-dimensional $\mathbb{Q}$-vector space $\mathfrak{h},$\vskip1mm
\hspace*{2.4em} (b) a finite index set $I$ (the set of simple roots),\vskip1mm
\hspace*{2.4em} (c) a linearly independent subset $\{\alpha_{i}\in \mathfrak{h}^{*}; i\in I\}$ of $\mathfrak{h}^{*},$ \vskip1mm
\hspace*{2.4em} (d) a subset $\{h_{i}; i\in I\}$ of $\mathfrak{h}$, a subset $\{\Lambda_{i}; i\in I\}$ of $\mathfrak{h}^{*}$ with $\langle h_{j}, \Lambda_{i}\rangle=\delta_{ij},$ \vskip1mm
\hspace*{2.4em} (e) a $\mathbb{Q}$-valued symmetric bilinear form $(\cdot,\cdot)$ on $\mathfrak{h}^{*},$\vskip1mm
\hspace*{2.4em} (f) a root lattice $Q$ and a weight lattice $P$ of $\mathfrak{h}^{*}.$

\noindent We assume that they satisfy the following properties:\vskip1mm
\hspace*{2.4em} (a) $\langle h_{i}, \lambda\rangle=\frac{2(\alpha_{i}, \lambda)}{(\alpha_{i}, \alpha_{i})}$ for any $i\in I$ and $\lambda\in \mathfrak{h}^{*}$.\vskip1mm
\hspace*{2.4em} (b) $(\alpha_{i}, \alpha_{i})\in 2\mathbb{Z}_{> 0}$, $\langle h_{i}, \alpha_{i}\rangle=2.$\vskip1mm
\hspace*{2.4em} (c) $\langle h_{i}, \alpha_{j}\rangle\in \mathbb{Z}_{\leq 0}$ for $i\neq j$ and $\langle h_{i}, \alpha_{j}\rangle=0\Leftrightarrow \langle h_{j}, \alpha_{i}\rangle=0.$\vskip1mm
\hspace*{2.4em} (d) $\alpha_{i}\in P$ and $h_{i}\in P^{*}=\{h\in \mathfrak{h}; \langle h, P\rangle\subset \mathbb{Z}\}$ for any $i.$\vskip1mm
\noindent Hence $\{\langle h_{i}, \alpha_{j}\rangle\}$ is a symmetrizable generalized Cartan matrix.

We also fix a matrix $\Lambda=(\Lambda_{ij})_{i,j\in I}$ with the following conditions: \vskip1mm
\hspace*{2.4em} (a) $\Lambda_{ii}\in \mathbb{Z}_{> 0},$ $\Lambda_{ij}\in \mathbb{Z}_{\leq 0}$ for all $i\neq j;$ \vskip1mm
\hspace*{2.4em} (b) $(\Lambda_{ij}+\Lambda_{ji})/\Lambda_{ii}\in \mathbb{Z}_{\leq 0}$ for all $i\neq j;$ \vskip1mm
\hspace*{2.4em} (c) the greatest common divisor of all $\Lambda_{ii}$ is equal to 1.\vskip1mm

\noindent To $\Lambda,$ we associate the following bilinear forms on $\mathbb{Z}^{I}:$ $$\langle i, j\rangle=\Lambda_{ij}, ~~~~i\cdot j=\langle i, j\rangle+\langle j, i\rangle.$$ Note that we have $$(\alpha_{i}, \alpha_{j})=i\cdot j, ~~~\langle h_{i}, \alpha_{j}\rangle=\frac{2i\cdot j}{i\cdot i}=(\Lambda_{ij}+\Lambda_{ji})/\Lambda_{ii} ~~\mathrm{for~ all}~ i,j\in I.$$

For $\lambda\in P,$ we linearly extend the bilinear form $\langle\cdot, \cdot\rangle$ to be defined on $P \times P$ such that $\langle\lambda, i\rangle=\langle\lambda, \alpha_i\rangle=\frac{1}{m}\sum\limits_{j=1}^{n}a_{j}\langle j, i\rangle,$ or $\langle i, \lambda\rangle=\langle \alpha_i, \lambda\rangle=\frac{1}{m}\sum\limits_{j=1}^{n}a_{j}\langle i, j\rangle$ for $\lambda=\frac{1}{m}\sum\limits_{j=1}^{n}a_{j}\alpha_j$ with $a_{j}\in \mathbb{Z},$ where $m$ is the possibly smallest positive integer such that $mP \subseteq Q.$

Let $\mathfrak{g}$ be the associated Kac-Moody Lie algebra. Then the two-parameter $(v,t)$-quantum algebra $U_{v,t}(\mathfrak{g})$ associated to $\mathfrak{g}$ is a unital associative $\mathbb{Q}(v,t)$-algebra generated by the elements $e_{i},$ $f_{i},$ $k_{i}^{\pm 1},$ $k_{i}'^{\pm 1}$ ($i\in I$) with the following fundamental relations: \vskip3mm
\hspace*{2.4em} (R1) $k_{i}^{\pm 1}k_{j}^{\pm 1}=k_{j}^{\pm 1}k_{i}^{\pm 1},~~~k_{i}'^{\pm 1}k_{j}'^{\pm 1}=k_{j}'^{\pm 1}k_{i}'^{\pm 1},$\vskip3mm
\hspace*{4.8em}      $k_{i}^{\pm 1}k_{j}'^{\pm 1}=k_{j}'^{\pm 1}k_{i}^{\pm 1},~~~k_{i}^{\pm 1}k_{i}^{\mp 1}=1=k_{i}'^{\pm 1}k_{i}'^{\mp 1};$\vskip3mm
\hspace*{2.4em} (R2) $k_{i}e_{j}k_{i}^{-1}=v^{i\cdot j}t^{\langle i, j\rangle-\langle j, i\rangle}e_{j},~~k_{i}'e_{j}k_{i}'^{-1}=v^{-i\cdot j}t^{\langle i, j\rangle-\langle j, i\rangle}e_{j},$\vskip3mm
\hspace*{4.8em}     $k_{i}f_{j}k_{i}^{-1}=v^{-i\cdot j}t^{\langle j, i\rangle-\langle i, j\rangle}f_{j},~~k_{i}'f_{j}k_{i}'^{-1}=v^{i\cdot j}t^{\langle j, i\rangle-\langle i, j\rangle}f_{j};$\vskip3mm
\hspace*{2.4em} (R3) $e_{i}f_{j}-f_{j}e_{i}=\delta_{ij}\frac{k_{i}-k_{i}'}{v_{i}-v_{i}^{-1}};$\vskip3mm
\hspace*{2.4em} (R4) $\sum\limits_{p+p'=1-\frac{2i\cdot j}{i\cdot i}}(-1)^{p}t_{i}^{-p\big(p'-2\frac{\langle i, j\rangle}{i\cdot i}+2\frac{\langle j, i\rangle}{i\cdot i}\big)}e_{i}^{[p']}e_{j}e_{i}^{[p]}=0$~~~if $i\neq j,$ \vskip3mm
\hspace*{4.8em} $\sum\limits_{p+p'=1-\frac{2i\cdot j}{i\cdot i}}(-1)^{p}t_{i}^{-p\big(p'-2\frac{\langle i, j\rangle}{i\cdot i}+2\frac{\langle j, i\rangle}{i\cdot i}\big)}f_{i}^{[p]}f_{j}f_{i}^{[p']}=0$~~~if $i\neq j,$\vskip3mm

\noindent where $v_{i}=v^{i\cdot i/2},$ $t_{i}=t^{i\cdot i/2},$ $e_{i}^{[p]}=\frac{e_{i}^{p}}{[p]_{v_{i},t_{i}}^{!}},$ $f_{i}^{[p]}=\frac{f_{i}^{p}}{[p]_{v_{i},t_{i}}^{!}}.$ The algebra $U_{v,t}(\mathfrak{g})$ has a Hopf algebra structure with the comultiplication $\Delta,$ the counit $\varepsilon$ and the antipode $S$ given as follows:\vskip3mm
\hspace*{2.4em} $\Delta(k_{i}^{\pm 1})=k_{i}^{\pm 1}\otimes k_{i}^{\pm 1},~~~\Delta(k_{i}'^{\pm 1})=k_{i}'^{\pm 1}\otimes k_{i}'^{\pm 1},$\vskip3mm
\hspace*{2.4em} $\Delta(e_{i})=e_{i}\otimes k_{i}'+ 1\otimes e_{i},~~~\Delta(f_{i})=f_{i}\otimes 1+ k_{i}\otimes f_{i},$\vskip3mm
\hspace*{2.4em} $\varepsilon(k_{i}^{\pm 1})=\varepsilon(k_{i}'^{\pm 1})=1,~~~\varepsilon(e_{i})=\varepsilon(f_{i})=0,~~~S(k_{i}^{\pm 1})=k_{i}^{\mp 1},$\vskip3mm
\hspace*{2.4em} $S(k_{i}'^{\pm 1})=k_{i}'^{\mp 1},~~~S(e_{i})=-e_{i}k_{i}'^{-1},~~~~S(f_{i})=-k_{i}^{-1}f_{i}.$\vskip3mm
Let $U_{v,t}^{+}(\mathfrak{g})$ (resp. $U_{v,t}^{-}(\mathfrak{g})$) be the subalgebra of $U_{v,t}(\mathfrak{g})$ generated by the elements $e_{i}$ (resp. $f_{i}$) for $i\in I,$ and let $U_{v,t}^{0}(\mathfrak{g})$ be the subalgebra of $U_{v,t}(\mathfrak{g})$ generated by $k_{i}^{\pm 1},$ $k_{i}'^{\pm 1}$ for $i\in I.$ Moreover, let $U_{v,t}^{\geq}$ (resp. $U_{v,t}^{\leq}$) be the subalgebra of $U_{v,t}(\mathfrak{g})$ generated by $e_{i},$ $k_{i}^{\pm 1}$ (resp. $f_{i},$ $k_{i}'^{\pm 1}$) for $i\in I.$

In analogy with the one-parameter case, we can prove the following proposition.
\begin{proposition}
{\rm (See [HP1, Proposition 2.4].)} There exists a unique bilinear form $\langle \cdot, \cdot\rangle: U_{v,t}^{\geq}\times U_{v,t}^{\leq}\longrightarrow\mathbb{Q}(v,t)$ such that for all $x,x'\in U_{v,t}^{\geq},$ $y,y'\in U_{v,t}^{\leq},$ $\mu,\nu\in Q$ and $i,j\in I,$ we have \vskip3mm
$\langle x, yy'\rangle=\langle \Delta(x), y\otimes y'\rangle,~~\langle xx', y\rangle=\langle x'\otimes x, \Delta(y)\rangle,~~\langle e_{i}, f_{j}\rangle=\delta_{ij}\frac{1}{v_{i}^{-1}-v_{i}},$\vskip3mm
$\langle k_{\nu}, k_{\mu}'\rangle=v^{\mu\cdot\nu}t^{\langle \nu, \mu\rangle-\langle \mu, \nu\rangle},~~\langle e_{i}, k_{\mu}'\rangle=\langle k_{\mu}, f_{i}\rangle=0,$\vskip3mm
\noindent where $k_{\mu}=\prod_{i\in I}k_{i}^{\mu_{i}},$ $k_{\mu}'=\prod_{i\in I}k_{i}'^{\mu_{i}}$ for each $\mu=\sum_{i\in I}\mu_{i}\alpha_{i}\in Q.$
\end{proposition}
Based on Proposition 2.1, with the similar argument of Theorem 2.5 and Corollary 2.6 in [BGH1], we have
\begin{corollary}{\rm (See [HP1, Corollary 2.5 and 2.6].)}
$U_{v,t}(\mathfrak{g})$ can be realized as a Drinfel'd double of Hopf subalgebras $U_{v,t}^{\geq}$ and $U_{v,t}^{\leq}$ with respect to the pairing $\langle \cdot, \cdot\rangle.$ Moreover, $U_{v,t}(\mathfrak{g})$ has the standard triangular decomposition $U_{v,t}(\mathfrak{g})\cong U_{v,t}^{-}(\mathfrak{g})\otimes U_{v,t}^{0}(\mathfrak{g})\otimes U_{v,t}^{+}(\mathfrak{g}).$
\end{corollary}

\subsection{Automorphisms of $U_{v,t}(\mathfrak{g})$}

There exists a unique automorphism $\omega$ of $U_{v,t}(\mathfrak{g})$ satisfying
$$\omega(e_{i})=f_{i},~\omega(f_{i})=e_{i},~\omega(k_{i})=k_{i}',~\omega(k_{i}')=k_{i},~\omega(v)=v,~\omega(t)=t^{-1}.$$

There exists a unique anti-automorphism $*$ of $U_{v,t}(\mathfrak{g})$ satisfying
$$e_{i}^{*}=e_{i},~f_{i}^{*}=f_{i},~k_{i}^{*}=k_{i}',~k_{i}'^{*}=k_{i},~v^{*}=v,~t^{*}=t^{-1}.$$

There exists a unique automorphism $-$ of $U_{v,t}(\mathfrak{g})$ satisfying
$$\overline{e_{i}}=e_{i},~\overline{f_{i}}=f_{i},~\overline{k_{i}}=k_{i}',~\overline{k_{i}'}=k_{i},~\overline{v}=v^{-1},~\overline{t}=t.$$

There exists a unique $\mathbb{Q}(v,t)$-automorphism $\sigma$ of $U_{v,t}(\mathfrak{g})$ satisfying
$$\sigma(e_{i})=-v_{i}^{-1}e_{i},~\sigma(f_{i})=-v_{i}f_{i},~\sigma(k_{i})=k_{i},~\sigma(k_{i}')=k_{i}'.$$

Let $\tau=\sigma\circ \omega \circ S.$ Then $\tau$ is the unique $\mathbb{Q}(v)$-anti-automorphism of $U_{v,t}(\mathfrak{g})$ satisfying
$$\tau(e_{i})=v_{i}^{-1}k_{i}^{-1}f_{i},~\tau(f_{i})=v_{i}^{-1}k_{i}'^{-1}e_{i},~\tau(k_{i})=k_{i}'^{-1},~\tau(k_{i}')=k_{i}^{-1},~\tau(t)=t^{-1}.$$

Furthermore, one can easily check (by evaluating at the generators) that $$\tau^{2}=1~~~~\mathrm{and}~~~~\Delta\circ\tau=(\tau\otimes\tau)\circ\Delta.$$

\subsection{Integrable representations}

Let $M$ be a $U_{v,t}(\mathfrak{g})$-module. For any $\lambda\in P,$ we set $$M_{\lambda}=\{y\in M;~k_{i}y=v^{i\cdot \lambda}t^{\langle i, \lambda\rangle-\langle \lambda, i\rangle}y,~k_{i}'y=v^{-i\cdot \lambda}t^{\langle i, \lambda\rangle-\langle \lambda, i\rangle}y,~\forall i\in I\}.$$
We say that $M$ is integrable if $M$ satisfies the following conditions:\vskip2mm
(1) $M$ has a weight space decomposition $M=\bigoplus_{\lambda\in P}M_{\lambda}$ and $\dim M_{\lambda}<\infty$ for any $\lambda\in P$.

(2) There exists a finite subset $F$ of $P$ such that wt$(M)\subset F+Q_{-},$ where $Q_{-}=\sum_{i}\mathbb{Z}_{\leq 0}\alpha_{i}.$

(3) For any $i,$ $e_{i}$ and $f_{i}$ are locally nilpotent on $M.$\vskip2mm

Let $P_{+}=\{\lambda\in P; \langle h_{i}, \lambda\rangle\geq 0 ~\mathrm{for~any}~i\in I\}$ and let $\mathcal{O}_{\mathrm{int}}$ denote the category of integrable $U_{v,t}(\mathfrak{g})$-modules.

Lusztig [L4, Theorem 6.2.2] gave the complete reducibility theorem in the quantum case, which was inspired by the proof of the analogous result in the Kac-Moody algebra case (see [Kac]). Fan and Li recently showed that the categories of weight modules of this two-parameter quantum algebra and the ordinary quantum algebra (letting $t=1$) are equivalent (see [FL2, Theorem 4.1(d) and Corollary 4.3]). Moreover, the integrable modules are preserved (see [FL2, Appendix A]). Thus one can obtain that $\mathcal{O}_{\mathrm{int}}$ is a semisimple category and its irreducible objects are isomorphic to some $V(\lambda)$ for some $\lambda\in P_{+},$ where $V(\lambda)$ is the irreducible $U_{v,t}(\mathfrak{g})$-module of highest weight $\lambda$ with highest weight vector $y_{\lambda},$ which is defined by
\begin{align*}
V(\lambda)&=\mathbf{U}/I_{\lambda}\\
&\cong U_{v,t}^{-}(\mathfrak{g})/(\sum\limits_{i}U_{v,t}^{-}(\mathfrak{g})f_{i}^{1+\langle h_{i}, \lambda\rangle}),
\end{align*}
\noindent where $\mathbf{U}=U_{v,t}(\mathfrak{g}),$ and
$$I_{\lambda}=\sum_{i}\mathbf{U}e_{i}+\sum_{\mu}\mathbf{U}(k_{\mu}-v^{\mu\cdot \lambda}t^{\langle \mu, \lambda\rangle-\langle \lambda, \mu\rangle})+\sum_{\mu}\mathbf{U}(k_{\mu}'-v^{-\mu\cdot \lambda}t^{\langle \mu, \lambda\rangle-\langle \lambda, \mu\rangle})+\sum\limits_{i}\mathbf{U}f_{i}^{1+\langle h_{i}, \lambda\rangle};$$ see also [BGH2, Theorem 3.15] and [PHR, Proposition 41] for the related work.

\section{Crystal bases of integrable $U_{v,t}(\mathfrak{g})$-modules}
In this section, we will develop the crystal basis theory for $U_{v,t}(\mathfrak{g})$-modules in the category $\mathcal{O}_{\mathrm{int}}.$
\subsection{Crystal bases}
Let $U_{v,t}(\mathfrak{g_{i}})$ be the subalgebra generated by $e_{i},$ $f_{i},$ $k_{i}^{\pm 1},$ $k_{i}'^{\pm 1}$ for some $i.$ Then it is obvious that $U_{v,t}(\mathfrak{g_{i}})\cong U_{v,t}(\mathfrak{sl_{2}})\cong U_{v}(\mathfrak{sl_{2}})\otimes_{\mathbb{Q}(v)}\mathbb{Q}(v,t)$ over $\mathbb{Q}(v,t).$ By the theory of integrable representations of $U_{v}(\mathfrak{sl_{2}})$ over $\mathbb{Q}(v)$, for an integrable $U_{v,t}(\mathfrak{g})$-module $M,$ we have $$M_{\lambda}=\bigoplus_{0\leq n\leq \langle h_{i}, \lambda+n\alpha_{i}\rangle}f_{i}^{(n)}(\mathrm{Ker}~e_{i}\cap M_{\lambda+n\alpha_{i}}),~~\mathrm{where}~f_{i}^{(n)}=\frac{f_{i}^{n}}{[n]_{v_{i}}^{!}}.$$
This means that for any weight vector $y\in M_{\lambda},$ there exists a unique family $\{y_{n}\}_{n\in \mathbb{Z}_{\geq 0}}$ of elements of $M$ such that $y=\sum_{n\geq 0}f_{i}^{(n)}y_{n}$ with $y_{n}\in \mathrm{Ker}~e_{i}\cap M_{\lambda+n\alpha_{i}}$ and $0\leq n\leq \langle h_{i}, \lambda+n\alpha_{i}\rangle.$ We call this expression the $i$-string decomposition of $y.$ For this $i$-string decomposition of $y,$ we define the Kashiwara operators $\tilde{e}_{i}$ and $\tilde{f}_{i}$ on $M$ by $$\tilde{e}_{i}(f_{i}^{(n)}y_{n})=f_{i}^{(n-1)}y_{n},~~~\tilde{f}_{i}(f_{i}^{(n)}y_{n})=f_{i}^{(n+1)}y_{n}.$$

Let $\mathbf{A}$ be the subring of $\mathbb{Q}(v,t)$ consisting of rational functions without poles at $v=0,$ that is, $\mathbf{A}$ is the localization of the ring $\mathbb{Q}(t)[v]$ at the prime ideal $\mathbb{Q}(t)[v]v.$ Thus the local ring $\mathbf{A}$ is a principal ideal domain with $\mathbb{Q}(v,t)$ as its field of quotients.\vskip2mm

\begin{definition} Let $M$ be an integrable $U_{v,t}(\mathfrak{g})$-module. A pair $(L, B)$ is called a (lower) crystal basis of $M$ if it satisfies the following conditions:\vskip2mm
(1) $L$ is a free $\mathbf{A}$-submodule of $M$ such that $M\cong \mathbb{Q}(v,t)\otimes_{\mathbf{A}}L;$

(2) $L=\bigoplus_{\lambda\in P}L_{\lambda},$ where $L_{\lambda}=L\cap M_{\lambda};$

(3) $\tilde{e}_{i}L\subset L, \tilde{f}_{i}L\subset L$ for any $i;$

(4) $B$ is a $\mathbb{Q}(t)$-basis of $L/vL\cong \mathbb{Q}(t)\otimes_{\mathbf{A}}L;$

(5) $B=\bigsqcup_{\lambda\in P}B_{\lambda},$ where $B_{\lambda}=B\cap L_{\lambda}/vL_{\lambda};$

(6) $\tilde{e}_{i}B\subset B\cup\{0\}$ and $\tilde{f}_{i}B\subset B\cup\{0\}$ for any $i;$

(7) for any $i$ and $b,b'\in B,$ $b'=\tilde{f}_{i}b$ if and only if $b=\tilde{e}_{i}b'.$
\end{definition}
If a free $\mathbf{A}$-submodule $L$ of $M$ satisfies (1), (2) and (3) as above, then we call it a crystal lattice.
\begin{lemma}{\rm (See [JKK, Lemma 4.5].)}
let $M\in \mathcal{O}_{\mathrm{int}},$ and let $L$ be a free $\mathbf{A}$-submodule of $M$ satisfying the conditions $(1)$ and $(2)$ in Definition 3.1. We have

$(a)$ $L$ is a crystal lattice if and only if $L=\bigoplus_{n\geq 0}f_{i}^{(n)}(\mathrm{Ker}~e_{i}\cap L)$ for every $i\in I.$

$(b)$ Let $\tilde{e}_{i}'$ and $\tilde{f}_{i}'$ be operators on $M$ such that $$\tilde{e}_{i}'(f_{i}^{(n)}y)=a_{\lambda,n}^{i}f_{i}^{(n-1)}y,~~~~\tilde{f}_{i}'(f_{i}^{(n)}y)=b_{\lambda,n}^{i}f_{i}^{(n+1)}y$$
for $y\in \mathrm{Ker}~e_{i}\cap M_{\lambda}$ with $f_{i}^{(n)}y\neq 0.$ Here, $a_{\lambda,n}^{i}$ and $b_{\lambda,n}^{i}$ are invertible elements of $\mathbf{A}.$ Then $L$ is a crystal lattice if and only if $\tilde{e}_{i}'L\subset L$ and $\tilde{f}_{i}'L\subset L.$ Moreover, if we define the operators $R_{i}$ and $S_{i}$ by $$R_{i}(f_{i}^{(n)}y)=a_{\lambda,n}^{i}(v=0)f_{i}^{(n)}y,~~~~S_{i}(f_{i}^{(n)}y)=b_{\lambda,n}^{i}(v=0)f_{i}^{(n)}y,$$
then $R_{i}L\subset L, S_{i}L\subset L$ and the induced actions of $\tilde{e}_{i}'$ and $\tilde{f}_{i}'$ on $L/vL$ coincide with those of $\tilde{e}_{i}\circ R_{i}$ and $\tilde{f}_{i}\circ S_{i},$ respectively. In particular, we have $\tilde{e}_{i}'^{-1}(L)=\tilde{e}_{i}^{-1}(L).$
\end{lemma}
\begin{lemma}
Let $M$ be an integrable $U_{v,t}(\mathfrak{g})$-module and let $(L, B)$ be a crystal basis of $M.$ For $i\in I$ and $y\in L_{\lambda},$ let $y=\sum_{n\geq 0}f_{i}^{(n)}y_{n}$ be the $i$-string decomposition of $y.$ Then the following statements are true.

$(a)$ $y_{n}\in L$ for all $n\geq 0.$

$(b)$ If $y+vL\in B,$ then there exists a non-negative integer $n_{0}$ such that $y_{n}\in vL$ for each $n\neq n_{0},$ $y_{n_{0}}\in B$ modulo $vL,$ and $y\equiv f_{i}^{(n_{0})}y_{n_{0}}~\mathrm{mod}~vL.$ In particular, we have $\tilde{e}_{i}y\equiv f_{i}^{(n_{0}-1)}y_{n_{0}}$ and $\tilde{f}_{i}y\equiv f_{i}^{(n_{0}+1)}y_{n_{0}}~\mathrm{mod}~vL.$
\end{lemma}
\subsection{Tensor product rule}
Let $M$ be an integrable $U_{v,t}(\mathfrak{g})$-module with a crystal basis $(L, B)$. For $i\in I$ and $b\in B,$ we set $$\varepsilon_{i}(b)=\mathrm{max}\{n; \tilde{e}_{i}^{n}b\neq 0\}=\mathrm{max}\{n; b\in \tilde{f}_{i}^{n}B\},$$$$\varphi_{i}(b)=\mathrm{max}\{n; \tilde{f}_{i}^{n}b\neq 0\}=\mathrm{max}\{n; b\in \tilde{e}_{i}^{n}B\}.$$
\begin{theorem}
Let $M_{j}$ be an integrable $U_{v,t}(\mathfrak{g})$-module and let $(L_{j}, B_{j})$ be a crystal basis of $M_{j}$ $(j=1,2).$ Set $M=M_{1}\otimes M_{2},$ $L=L_{1}\otimes_{\mathbf{A}}L_{2},$ $B=B_{1}\otimes B_{2}=\{b_{1}\otimes b_{2}; b_{j}\in B_{j}\}.$ Then we have

$(a)$ $(L, B)$ is a crystal basis of $M.$

$(b)$ For $b_{1}\in B_{1},$ $b_{2}\in B_{2}$ and $i\in I,$ we have\[\tilde{f}_{i}(b_{1}\otimes b_{2})=\begin{cases}\tilde{f}_{i}b_{1}\otimes b_{2}& \mathrm{if}~ \varphi_{i}(b_1)> \varepsilon_{i}(b_2), \\ b_{1}\otimes\tilde{f}_{i}b_{2}& \mathrm{if}~
\varphi_{i}(b_1)\leq \varepsilon_{i}(b_2);
\end{cases}\]
\[\tilde{e}_{i}(b_{1}\otimes b_{2})=\begin{cases}\tilde{e}_{i}b_{1}\otimes b_{2}& \mathrm{if}~ \varphi_{i}(b_1)\geq \varepsilon_{i}(b_2), \\ b_{1}\otimes\tilde{e}_{i}b_{2}& \mathrm{if}~
\varphi_{i}(b_1)< \varepsilon_{i}(b_2).
\end{cases}\]
\end{theorem}
\begin{proof}
By Lemma 3.3, it suffices to prove the tensor product rule for irreducible highest weight $U_{v,t}(\mathfrak{g_{i}})$-modules, thus it is reduced to the $U_{v}(\mathfrak{sl_{2}})$ case; see [Kas1, Proposition 6] and [Kas2, Theorem 1].
\end{proof}

\begin{proposition}
Let $M$ be an integrable $U_{v,t}(\mathfrak{g})$-module with a crystal basis $(L, B)$. For any $n\in \mathbb{Z}_{\geq 0}$ and $i\in I,$ we have $$(f_{i}^{n}M\cap L)/(f_{i}^{n}M\cap qL)=\bigoplus_{\substack{b\in B\\\varepsilon_{i}(b)\geq n}}\mathbb{Q}(t)b.$$ In particular, we have $\dim_{\mathbb{Q}(v,t)}(f_{i}^{n}M)_{\lambda}=\#\{b\in B_{\lambda}; \varepsilon_{i}(b)\geq n\}.$
\end{proposition}
\subsection{Polarization}
Consider the anti-involution $\tau$ on $U_{v,t}(\mathfrak{g})$. By the standard arguments, we have the following proposition.
\begin{proposition}
Let $\lambda\in P_{+}.$ There exists a unique nondegenerate symmetric bilinear form $(\cdot, \cdot)$ on $V(\lambda)$ satisfying $(y_{\lambda},y_{\lambda})=1,$ and $$(k_{i}x,y)=(x, k_{i}'^{-1}y), ~~(k_{i}'x,y)=(x, k_{i}^{-1}y),~~(e_{i}x,y)=(x, v_{i}^{-1}k_{i}^{-1}f_{i}y),$$ $$(f_{i}x,y)=(x, v_{i}^{-1}k_{i}'^{-1}e_{i}y)~~\mathrm{for~}i\in I~\mathrm{and}~x,y\in V(\lambda).\eqno{(3.1)}$$
\end{proposition}
For an integrable $U_{v,t}(\mathfrak{g})$-module $M,$ we call a bilinear form $(\cdot, \cdot)$ on $M$ a polarization if (3.1) is satisfied with $M$ in place of $V(\lambda).$ We can easily get the following lemma using the properties of $\tau.$
\begin{lemma}
Assume $M,N\in \mathcal{O}_{\mathrm{int}}$ admit polarizations $(\cdot, \cdot).$ Then the module $M\otimes N,$ on which the symmetric bilinear form is given by $(m_{1}\otimes n_{1},  m_{2}\otimes n_{2})=(m_1,m_2)(n_1,n_2),$ also admits a polarization.
\end{lemma}
For $\lambda,\mu\in P_{+},$ there exist unique $U_{v,t}(\mathfrak{g})$-module homomorphisms $$\Phi_{\lambda, \mu}: V(\lambda+\mu)\rightarrow V(\lambda)\otimes V(\mu),~~~~ \Psi_{\lambda, \mu}: V(\lambda)\otimes V(\mu)\rightarrow V(\lambda+\mu)$$ satisfying $$\Phi_{\lambda, \mu}(y_{\lambda+\mu})=y_{\lambda}\otimes y_{\mu},~~~~ \Psi_{\lambda, \mu}(y_{\lambda}\otimes y_{\mu})=y_{\lambda+\mu}.$$ It is clear that $\Psi_{\lambda, \mu}\circ \Phi_{\lambda, \mu}=\mathrm{id}_{V(\lambda+\mu)}$ and they commute with $\tilde{e}_{i}$ and $\tilde{f}_{i}.$ Moreover, we have $$(\Psi_{\lambda, \mu}(x), y)=(x, \Phi_{\lambda, \mu}(y)) ~~~\mathrm{for}~ x\in V(\lambda)\otimes V(\mu) ~\mathrm{and}~ y\in V(\lambda+\mu).$$
This follows easily from the uniqueness of a bilinear form $(\cdot, \cdot)$ on $(V(\lambda)\otimes V(\mu))\times V(\lambda+\mu)$ satisfying (3.1) and $(y_{\lambda}\otimes y_{\mu}, y_{\lambda+\mu})=1.$

\section{Crystal bases of $U_{v,t}^{-}(\mathfrak{g})$}
In this section, we shall define the notion of crystal basis for $U_{v,t}^{-}(\mathfrak{g})$ following [Kas2].
\subsection{Two-parameter Kashiwara algebras}
Let $\mathcal{B}_{v,t}(\mathfrak{g})$ be the algebra over $\mathbb{Q}(v,t)$ generated by the elements $e_{i}', f_{i}$ $(i\in I)$ with the following relations:
$$e_{i}'f_{j}=v^{-i\cdot j}t^{\langle i, j\rangle-\langle j, i\rangle}f_{j}e_{i}'+\delta_{ij};\eqno{(4.1)}$$
$$\sum\limits_{p+p'=1-\frac{2i\cdot j}{i\cdot i}}(-1)^{p}t_{i}^{-p\big(p'-2\frac{\langle i, j\rangle}{i\cdot i}+2\frac{\langle j, i\rangle}{i\cdot i}\big)}X_{i}^{[p]}X_{j}X_{i}^{[p']}=0~~for~ i\neq j,~X_{i}=e_{i}', f_{i}.\eqno{(4.2)}$$
We call $\mathcal{B}_{v,t}(\mathfrak{g})$ the two-parameter Kashiwara algebra. $\mathcal{B}_{v,t}(\mathfrak{g})$ has an antiautomorphism $\rho$ defined by $\rho(f_{i})=e_{i}',$ $\rho(e_{i}')=f_{i}$ ($i\in I$) and $\rho(v)=v,$ $\rho(t)=t^{-1}.$

\begin{lemma}
For $i\in I,$ there exist unique $e_{i}'$ and $e_{i}''$ in $\mathrm{End}(U_{v,t}^{-}(\mathfrak{g}))$ such that for any $y\in U_{v,t}^{-}(\mathfrak{g}),$ we have $$[e_{i}, y]=\frac{k_{i}e_{i}''(y)-k_{i}'e_{i}'(y)}{v_{i}-v_{i}^{-1}}.$$
\end{lemma}
\begin{proof}
The uniqueness follows from the triangular decomposition in Corollary 2.2. Since $U_{v,t}^{-}(\mathfrak{g})$ is generated by the $f_{j}$ and the lemma is true for $y=1,$ it is enough to show that the lemma is true for $f_{j}y$ if it is true for $y.$ We have
$$[e_{i}, f_{j}y]=[e_{i}, f_{j}]y+f_{j}[e_{i}, y]=\delta_{ij}\frac{k_{i}-k_{i}'}{v_{i}-v_{i}^{-1}}y+f_{j}\frac{k_{i}e_{i}''(y)-k_{i}'e_{i}'(y)}{v_{i}-v_{i}^{-1}}.$$
It follows from the definition that we have $$[e_{i}, f_{j}y]=\frac{k_{i}(v^{i\cdot j}t^{\langle i, j\rangle-\langle j, i\rangle} f_{j}e_{i}''(y)+\delta_{ij}y)-k_{i}'(v^{-i\cdot j}t^{\langle i, j\rangle-\langle j, i\rangle}f_{j}e_{i}'(y)+\delta_{ij}y)}{v_{i}-v_{i}^{-1}}.$$
\end{proof}
\begin{proposition}
For any $i,j\in I,$ we have $e_{i}'e_{j}''=v^{i\cdot j}t^{-\langle i, j\rangle+\langle j, i\rangle}e_{j}''e_{i}'$ in $\mathrm{End}(U_{v,t}^{-}(\mathfrak{g})).$
\end{proposition}
\begin{proof}
For $k\in I,$ we have
\begin{align*}
e_{i}'e_{j}''f_{k}&=e_{i}'(v^{j\cdot k}t^{\langle j, k\rangle-\langle k, j\rangle}f_{k}e_{j}''+\delta_{jk})\\&=v^{j\cdot k}t^{\langle j, k\rangle-\langle k, j\rangle}(v^{-i\cdot k}t^{\langle i, k\rangle-\langle k, i\rangle}f_{k}e_{i}'+\delta_{ik})e_{j}''+\delta_{jk}e_{i}'\\&=v^{j\cdot k-i\cdot k}t^{\langle i+j, k\rangle-\langle k, i+j\rangle}f_{k}e_{i}'e_{j}''+v^{j\cdot i}t^{\langle j, i\rangle-\langle i, j\rangle}\delta_{ik}e_{j}''+\delta_{jk}e_{i}'.
\end{align*}
Similarly, we have
\begin{align*}
e_{j}''e_{i}'f_{k}&=e_{j}''(v^{-i\cdot k}t^{\langle i, k\rangle-\langle k, i\rangle}f_{k}e_{i}'+\delta_{ik})\\&=v^{-i\cdot k}t^{\langle i, k\rangle-\langle k, i\rangle}(v^{j\cdot k}t^{\langle j, k\rangle-\langle k, j\rangle}f_{k}e_{j}''+\delta_{jk})e_{i}'+\delta_{ik}e_{j}''\\&=v^{j\cdot k-i\cdot k}t^{\langle i+j, k\rangle-\langle k, i+j\rangle}f_{k}e_{j}''e_{i}'+v^{-i\cdot j}t^{\langle i, j\rangle-\langle j, i\rangle}\delta_{jk}e_{i}'+\delta_{ik}e_{j}''.
\end{align*}
Hence, if we set $S=e_{i}'e_{j}''-v^{i\cdot j}t^{\langle j, i\rangle-\langle i, j\rangle}e_{j}''e_{i}',$ then $Sf_{k}=v^{j\cdot k-i\cdot k}t^{\langle i+j, k\rangle-\langle k, i+j\rangle}f_{k}S.$ Then $S\cdot 1=0$ gives $S=0.$
\end{proof}

For any $\xi=\sum_{i}n_{i}\alpha_{i}\in Q_{-},$ we write $|\xi|=\sum_{i}|n_{i}|,$ and we set $$U_{v,t}^{-}(\mathfrak{g})_{\xi}=\{y\in U_{v,t}^{-}(\mathfrak{g}); k_{i}yk_{i}^{-1}=v^{i\cdot \xi}t^{-\langle \xi, i\rangle+\langle i, \xi\rangle}y,~k_{i}'yk_{i}'^{-1}=v^{-i\cdot \xi}t^{-\langle \xi, i\rangle+\langle i, \xi\rangle}y,~\forall i\in I\}.$$
\begin{corollary}{\rm (See also [FL, Lemma 9].)}
Let $y\in U_{v,t}^{-}(\mathfrak{g}).$ If $e_{i}'(y)=0$ for all $i\in I,$ then $y$ is a constant multiple of $1.$
\end{corollary}
\begin{proof}
We may assume $y\in U_{v,t}^{-}(\mathfrak{g})_{\xi}$ for some $\xi\in Q_{-}.$ We shall prove it by the induction on $|\xi|.$ We may assume $\xi\neq 0.$

$(a)$ If $|\xi|=1.$ In this case, $y=cf_{i}$ for some $i$ and $c\in \mathbb{Q}(v,t).$ Therefore $c=e_{i}'y=0.$

$(b)$ If $|\xi|>1.$ For any $j\in I,$ we have $e_{i}'e_{j}''(y)=v^{i\cdot j}t^{-\langle i, j\rangle+\langle j, i\rangle}e_{j}''e_{i}'(y)=0$. Hence $e_{j}''(y)=0$ by the hypothesis of induction. Hence, we have $e_{i}y=ye_{j}$ for any $j\in I.$ Now let $\lambda\in P_{+}$ satisfy $\langle h_{i}, \lambda\rangle\gg 0$ so that $U_{v,t}^{-}(\mathfrak{g})_{\xi}\xrightarrow{\sim}V(\lambda)_{\lambda+\xi}$ by the homomorphism $x\mapsto xu_{\lambda}.$ Since $V(\lambda)$ is irreducible and $U_{v,t}(\mathfrak{g})yu_{\lambda}$ does not contain $u_{\lambda},$ then $e_{j}(yu_{\lambda})=0$ for any $j\in I$ implies that $yu_{\lambda}=0,$ and hence $y=0.$
\end{proof}
\begin{corollary}
Let $i\in I$ and let $y\in U_{v,t}^{-}(\mathfrak{g})_{\xi},$ which satisfies $e_{i}'(y)=0.$ Then for any $U_{v,t}(\mathfrak{g})$-module $M$ and $m\in M_{\lambda}$ such that $e_{i}m=0,$ we have $$k_{i}'^{-n}e_{i}^{n}ym=\frac{v_{i}^{n(2\langle h_{i}, \lambda+\xi\rangle+3n+1)}}{(v_{i}-v_{i}^{-1})^{n}}(e_{i}''^{n}y)m.$$
\end{corollary}
\begin{proof}
We shall prove it by the induction on $n.$ We have
\begin{align*}
k_{i}'^{-n-1}e_{i}^{n+1}ym&=k_{i}'^{-1}k_{i}'^{-n}e_{i}e_{i}^{n}ym\\
&=v^{n i\cdot i}k_{i}'^{-1}e_{i}k_{i}'^{-n}e_{i}^{n}ym\\
&=v^{2ni\cdot(\lambda+\xi)}v_{i}^{n(3n+1)}(v_{i}-v_{i}^{-1})^{-n}v^{n i\cdot i}k_{i}'^{-1}e_{i}(e_{i}''^{n}y)m.
\end{align*}
Since $$k_{i}'^{-1}e_{i}(e_{i}''^{n}y)m=k_{i}'^{-1}[e_{i}, e_{i}''^{n}y]m=\frac{k_{i}'^{-1}k_{i}e_{i}''^{n+1}y-k_{i}'^{-1}k_{i}'e_{i}'e_{i}''^{n}y}{v_{i}-v_{i}^{-1}}m.$$
By Proposition 4.2 we have $e_{i}'e_{i}''^{n}(y)=0.$ Hence we obtain
\begin{align*}
k_{i}'^{-n-1}e_{i}^{n+1}ym&=v^{2ni\cdot(\lambda+\xi)}v_{i}^{n(3n+3)}(v_{i}-v_{i}^{-1})^{-n-1}k_{i}'^{-1}k_{i}(e_{i}''^{n+1}y)m\\
&=v^{2ni\cdot(\lambda+\xi)}v_{i}^{n(3n+3)}(v_{i}-v_{i}^{-1})^{-n-1}v^{2(n+1)i\cdot i}v^{2i\cdot \xi}v^{2i\cdot \lambda}(e_{i}''^{n+1}y)m\\
&=v_{i}^{(n+1)(2\langle h_{i}, \lambda+\xi\rangle+3n+4)}(v_{i}-v_{i}^{-1})^{-n-1}(e_{i}''^{n+1}y)m.
\end{align*}
\end{proof}

Let $f_{j}$ ($j\in I$) acts on $U_{v,t}^{-}(\mathfrak{g})$ by the left multiplication. Our interest in $\mathcal{B}_{v,t}(\mathfrak{g})$ comes from the following result.
\begin{lemma}
$U_{v,t}^{-}(\mathfrak{g})$ is a $\mathcal{B}_{v,t}(\mathfrak{g})$-module.
\end{lemma}
\begin{proof}
Let $b=1-\frac{2i\cdot j}{i\cdot i}.$ From Lemma 4.1, it remains to prove that for $i\neq j,$ $$S=\sum\limits_{p+p'=b}(-1)^{p}t_{i}^{-p\big(p'-2\frac{\langle i, j\rangle}{i\cdot i}+2\frac{\langle j, i\rangle}{i\cdot i}\big)}e_{i}'^{[p]}e_{j}'e_{i}'^{[p']}$$ vanishes as an endomorphism of $U_{v,t}^{-}(\mathfrak{g})$.
From (4.1) we have $$e_{i}'^{n}f_{j}=v^{-n i\cdot j}t^{n(\langle i, j\rangle-\langle j, i\rangle)}f_{j}e_{i}'^{n}+\delta_{ij}v_{i}^{1-n}[n]_{v_{i}}e_{i}'^{n-1}.$$
Hence we have
\begin{align*}
Sf_{k}&=\sum\limits_{p+p'=b}\frac{(-1)^{p}}{[p']_{v_{i},t_{i}}^{!}}t_{i}^{-p\big(p'-2\frac{\langle i, j\rangle}{i\cdot i}+2\frac{\langle j, i\rangle}{i\cdot i}\big)}e_{i}'^{[p]}e_{j}'(v^{-p' i\cdot k}t^{p'(\langle i, k\rangle-\langle k, i\rangle)}f_{k}e_{i}'^{p'}+\delta_{ik}v_{i}^{1-p'}[p']_{v_{i}}e_{i}'^{p'-1})\\
&=\sum\limits_{p+p'=b}\frac{(-1)^{p}}{[p']_{v_{i},t_{i}}^{!}}t_{i}^{-p\big(p'-2\frac{\langle i, j\rangle}{i\cdot i}+2\frac{\langle j, i\rangle}{i\cdot i}\big)}e_{i}'^{[p]}\Big\{v^{-p' i\cdot k}t^{p'(\langle i, k\rangle-\langle k, i\rangle)}(v^{-j\cdot k}t^{\langle j, k\rangle-\langle k, j\rangle}f_{k}e_{j}'+\delta_{jk})e_{i}'^{p'}\\&~~\hspace*{0.4cm}+\delta_{ik}v_{i}^{1-p'}[p']_{v_{i}}e_{j}'e_{i}'^{p'-1}\Big\}\\
&=\sum\limits_{p+p'=b}\frac{(-1)^{p}}{[p']_{v_{i},t_{i}}^{!}[p]_{v_{i},t_{i}}^{!}}t_{i}^{-p\big(p'-2\frac{\langle i, j\rangle}{i\cdot i}+2\frac{\langle j, i\rangle}{i\cdot i}\big)}\Big\{v^{-p' i\cdot k}t^{p'(\langle i, k\rangle-\langle k, i\rangle)}v^{-j\cdot k}t^{\langle j, k\rangle-\langle k, j\rangle}\\
&~~\hspace*{0.4cm}\times(v^{-p i\cdot k}t^{p(\langle i, k\rangle-\langle k, i\rangle)}f_{k}e_{i}'^{p}+\delta_{ik}v_{i}^{1-p}[p]_{v_{i}}e_{i}'^{p-1})e_{j}'e_{i}'^{p'}+\delta_{jk}v^{-p' i\cdot k}t^{p'(\langle i, k\rangle-\langle k, i\rangle)}e_{i}'^{b}\\
&~~\hspace*{0.4cm}+\delta_{ik}v_{i}^{1-p'}[p']_{v_{i}}e_{i}'^{p}e_{j}'e_{i}'^{p'-1}\Big\}\\
\end{align*}
\begin{align*}&=v^{-b i\cdot k-j\cdot k}t^{b\langle i, k\rangle-b\langle k, i\rangle+\langle j, k\rangle-\langle k, j\rangle}f_{k}S\\
&~~\hspace*{0.4cm}+\sum\limits_{p+p'=b}\frac{(-1)^{p}}{[p']_{v_{i},t_{i}}^{!}[p]_{v_{i},t_{i}}^{!}}t_{i}^{-p\big(p'-2\frac{\langle i, j\rangle}{i\cdot i}+2\frac{\langle j, i\rangle}{i\cdot i}\big)}\delta_{ik}v^{-p' i\cdot i-i\cdot j}t^{\langle j, i\rangle-\langle i, j\rangle)}v_{i}^{1-p}[p]_{v_{i}}e_{i}'^{p-1}e_{j}'e_{i}'^{b-p}\\&~~\hspace*{0.4cm}+\sum\limits_{p+p'=b}\frac{(-1)^{p}}{[p']_{v_{i},t_{i}}^{!}[p]_{v_{i},t_{i}}^{!}}t_{i}^{-p\big(p'-2\frac{\langle i, j\rangle}{i\cdot i}+2\frac{\langle j, i\rangle}{i\cdot i}\big)}\delta_{ik}v_{i}^{1-p'}[p']_{v_{i}}e_{i}'^{p}e_{j}'e_{i}'^{b-p-1}\\&~~\hspace*{0.4cm}+\sum\limits_{p+p'=b}\frac{(-1)^{p}}{[p']_{v_{i},t_{i}}^{!}[p]_{v_{i},t_{i}}^{!}}t_{i}^{-p\big(p'-2\frac{\langle i, j\rangle}{i\cdot i}+2\frac{\langle j, i\rangle}{i\cdot i}\big)}v^{-p' i\cdot j}t^{p'(\langle i, j\rangle-\langle j, i\rangle)}e_{i}'^{b}.
\end{align*}
It is easy to see that the second term and the third one cancel out, the last term vanishes follows from $\sum\limits_{p+p'=b}\frac{(-1)^{p}}{[p']_{v_{i}}^{!}[p]_{v_{i}}^{!}}v_{i}^{-p'(1-b)}=0.$ Thus we have $$Sf_{k}=v^{-b i\cdot k-j\cdot k}t^{b\langle i, k\rangle-b\langle k, i\rangle+\langle j, k\rangle-\langle k, j\rangle}f_{k}S.$$ $S=0$ follows from the above identity and $S\cdot 1=0,$ since $e_{i}'\cdot 1=0$ for all $i\in I.$
\end{proof}

Furthermore, we can get the following result by Corollary 4.3.
\begin{corollary}
$U_{v,t}^{-}(\mathfrak{g})$ is a simple $\mathcal{B}_{v,t}(\mathfrak{g})$-module.
\end{corollary}

\begin{corollary}
$U_{v,t}^{-}(\mathfrak{g})\cong\mathcal{B}_{v,t}(\mathfrak{g})/\sum_{i}\mathcal{B}_{v,t}(\mathfrak{g})e_{i}'$ as $\mathcal{B}_{v,t}(\mathfrak{g})$-modules. \end{corollary}
\begin{proof}
Since $e_{i}'\cdot 1=0$ for all $i\in I,$ we have a surjective $\mathcal{B}_{v,t}(\mathfrak{g})$-module morphism $\mathcal{B}_{v,t}(\mathfrak{g})/\sum_{i}\mathcal{B}_{v,t}(\mathfrak{g})e_{i}'\rightarrow U_{v,t}^{-}(\mathfrak{g}).$ On the other hand, the $f_{i}$ ($i\in I$) generate a subalgebra of $\mathcal{B}_{v,t}(\mathfrak{g})$ isomorphic to $U_{v,t}^{-}(\mathfrak{g})$, so this map must be an isomorphism.
\end{proof}
\subsection{Polarization on $U_{v,t}^{-}(\mathfrak{g})$}
\begin{proposition}
There is a unique symmetric bilinear form $(\cdot,\cdot)$ $($called a polarization$)$ on $U_{v,t}^{-}(\mathfrak{g})$ such that $$(1,1)=1,~~~~(f_{i}x,y)=(x,e_{i}'y)~~\mathrm{for~any~} x,y\in U_{v,t}^{-}(\mathfrak{g})~\mathrm{and}~i\in I.$$
\end{proposition}
\begin{proof}
Let us endow $M=\mathrm{Hom}(U_{v,t}^{-}(\mathfrak{g}), \mathbb{Q}(v,t))$ with a structure of a left $\mathcal{B}_{v,t}(\mathfrak{g})$-module via $\rho;$ i.e., we have $(p\cdot\phi)(x)=\phi(\rho(p)\cdot x)$ for $p\in \mathcal{B}_{v,t}(\mathfrak{g}),$ $\phi\in M$ and $x\in U_{v,t}^{-}(\mathfrak{g}).$

Let $\phi_{0}\in M$ be defined by $\phi_{0}(1)=1$ and $\phi_{0}(\sum_{i}f_{i}U_{v,t}^{-}(\mathfrak{g}))=0.$ Since $e_{i}'\phi_{0}=0$ for any $i,$ we have a $\mathcal{B}_{v,t}(\mathfrak{g})$-homomorphism $\Phi: U_{v,t}^{-}(\mathfrak{g})\cong\mathcal{B}_{v,t}(\mathfrak{g})/\sum_{i}\mathcal{B}_{v,t}(\mathfrak{g})e_{i}'\rightarrow M,$ which sends $1$ to $\phi_{0}.$

Now we define a bilinear form $(\cdot, \cdot)$ on $U_{v,t}^{-}(\mathfrak{g})$ by $(x, y)=\Phi(x)(y)$ for $x, y\in U_{v,t}^{-}(\mathfrak{g}).$ Then we have $(1,1)=1,$ $(f_{i}x,y)=(f_{i}\Phi(x))(y)=(x,e_{i}'y).$ It is clear that these properties completely determine the bilinear form. Since the form $(\cdot, \cdot)'$ defined by $(x,y)'=\Phi(y)(x)$ satisfies the same properties, the symmetry follows from the uniqueness of such a bilinear form.
\end{proof}

\begin{corollary}
The bilinear form $(\cdot,\cdot)$ on $U_{v,t}^{-}(\mathfrak{g})$ is nondegenerate. Moreover, we have $(U_{v,t}^{-}(\mathfrak{g})_{\xi}, U_{v,t}^{-}(\mathfrak{g})_{\eta})=0$ for $\xi\neq \eta.$
\end{corollary}
\begin{proof}
The second claim follows from the definition of the bilinear form and can be shown by induction on the height of weights. Nondegeneracy of the bilinear form may be also shown by induction on $|\xi|.$ If $|\xi|=0,$ this is trivial. Assume $|\xi|> 0.$ If $y\in U_{v,t}^{-}(\mathfrak{g})_{\xi}$ satisfies $(y, U_{v,t}^{-}(\mathfrak{g})_{\xi})=0,$ then $(e_{i}'(y), U_{v,t}^{-}(\mathfrak{g})_{\xi+\alpha_{i}})=(y, f_{i}U_{v,t}^{-}(\mathfrak{g})_{\xi+\alpha_{i}})=0,$ and hence $e_{i}'(y)=0$ for any $i$ by the hypothesis of induction. It follows that $y=0$ by Corollary 4.3.
\end{proof}

\subsection{Crystal bases of $U_{v,t}^{-}(\mathfrak{g})$}
Let $\mathcal{O}(B_{v,t}(\mathfrak{g}))$ be the category of $B_{v,t}(\mathfrak{g})$-modules $M$ such that for any $m\in M$ there exists an integer $l$ such that $e_{i_{1}}'\cdots e_{i_{l}}'m=0$ for any $i_{1},\ldots, i_{l}\in I.$

Let $i\in I$ and let $M\in \mathcal{O}(B_{v,t}(\mathfrak{g}))$. Then by [Kas2, Proposition 3.2.1(a)] we have $$M=\bigoplus_{n\geq 0} f_{i}^{(n)}\mathrm{Ker}~e_{i}'.$$ We define the endomorphisms $$\tilde{e}_{i}(f_{i}^{(n)}y)=f_{i}^{(n-1)}y,~~~~\tilde{f}_{i}(f_{i}^{(n)}y)=f_{i}^{(n+1)}y~~\mathrm{for}~y\in\mathrm{Ker}~e_{i}'.$$

\begin{definition} Let $M\in \mathcal{O}(B_{v,t}(\mathfrak{g}))$. A pair $(L, B)$ is called a (lower) crystal basis of $M$ if it satisfies the following conditions:\vskip2mm
(1) $L$ is a free $\mathbf{A}$-submodule of $M$ such that $M\cong \mathbb{Q}(v,t)\otimes_{\mathbf{A}}L;$

(2) $\tilde{e}_{i}L\subset L,$ $\tilde{f}_{i}L\subset L$ for any $i;$

(3) $B$ is a $\mathbb{Q}(t)$-basis of $L/vL\cong \mathbb{Q}(t)\otimes_{\mathbf{A}}L;$

(4) $\tilde{e}_{i}B\subset B\cup\{0\}$ and $\tilde{f}_{i}B\subset B$ for any $i;$

(5) for any $i$ and $b\in B$ such that $\tilde{e}_{i}b\in B,$ we have $b=\tilde{f}_{i}\tilde{e}_{i}b.$
\end{definition}
\noindent $\mathbf{Remark ~4.11.}$ A multi-parameter version of the Kashiwara algebra can be found in [KKO, Definition 3.7].

\section{Existence of crystal bases}
In this section, we will prove the existence of crystal bases for integrable $U_{v,t}(\mathfrak{g})$-modules and for the algebra $U_{v,t}^{-}(\mathfrak{g})$. Assuming the existence theorem, one can prove the uniqueness of crystal bases by the same arguments as in the proof of [Kas2, Theorem 3].

For $\lambda\in P_{+},$ let $L(\lambda)$ be the $\mathbf{A}$-submodule of $V(\lambda)$ generated by the elements of the form $\tilde{f}_{i_{1}}\cdots \tilde{f}_{i_{l}}y_{\lambda},$ $l\geq 0,$ $i_{1},\ldots, i_{l}\in I,$ and let $B(\lambda)$ be the subset of $L(\lambda)/vL(\lambda)$ consisting of the nonzero vectors of the form $\tilde{f}_{i_{1}}\cdots \tilde{f}_{i_{l}}y_{\lambda},$ $l\geq 0,$ $i_{1},\ldots, i_{l}\in I.$ Similarly, we define $L(\infty)$ to be the $\mathbf{A}$-submodule of $U_{v,t}^{-}(\mathfrak{g})$ generated by the vectors of the form $\tilde{f}_{i_{1}}\cdots \tilde{f}_{i_{l}}\cdot 1,$ $l\geq 0,$ $i_{1},\ldots, i_{l}\in I,$ and set $B(\infty)$ to be the subset of $L(\infty)/vL(\infty)$ consisting of the nonzero vectors of the form $\tilde{f}_{i_{1}}\cdots \tilde{f}_{i_{l}}\cdot 1,$ $l\geq 0,$ $i_{1},\ldots, i_{l}\in I.$ In this section, our goal is to prove the following theorem.
\begin{theorem}
$(a)$ The pair $(L(\lambda), B(\lambda))$ is a crystal basis of $V(\lambda).$

$(b)$ The pair $(L(\infty), B(\infty))$ ia a crystal basis of $U_{v,t}^{-}(\mathfrak{g}).$

$(c)$ For $\lambda\in P_{+},$ let $\pi_{\lambda}: U_{v,t}^{-}(\mathfrak{g})\rightarrow V(\lambda)$ be the $U_{v,t}^{-}(\mathfrak{g})$-module homomorphism defined by $P\mapsto Py_{\lambda}.$ Then

$(\mathbf{i})$ $\pi_{\lambda}(L(\infty))=L(\lambda).$

~~~~~~Hence $\pi_{\lambda}$ induces a surjective homomorphism $\bar{\pi}_{\lambda}: L(\infty)/vL(\infty)\rightarrow L(\lambda)/vL(\lambda).$

$(\mathbf{ii})$ The set $\{b\in B(\infty); \bar{\pi}_{\lambda}(b)\neq 0\}$ is isomorphic to $B(\lambda).$

$(\mathbf{iii})$ $\tilde{f}_{i}\circ \bar{\pi}_{\lambda}=\bar{\pi}_{\lambda}\circ \tilde{f}_{i}.$

$(\mathbf{iv})$ If $b\in B(\infty)$ satisfies $\bar{\pi}_{\lambda}(b)\neq 0,$ then $\tilde{e}_{i}\bar{\pi}_{\lambda}(b)=\bar{\pi}_{\lambda}(\tilde{e}_{i}b).$
\end{theorem}

For $\lambda,\mu\in P_{+},$ recall the $U_{v,t}(\mathfrak{g})$-module homomorphisms $\Phi_{\lambda, \mu}$ and $\Psi_{\lambda, \mu}$, which commute with the Kashiwara operators $\tilde{e}_{i}$ and $\tilde{f}_{i}.$ We will define a map $S_{\lambda, \mu}: V(\lambda)\otimes V(\mu)\rightarrow V(\lambda)$ by $$S_{\lambda, \mu}(y\otimes y_{\mu})=y~~\mathrm{for}~ y\in V(\lambda) ~~~\mathrm{and}~~~S_{\lambda, \mu}(V(\lambda)\otimes \sum_{i}f_{i}V(\mu))=0.$$ Note that $S_{\lambda, \mu}$ is $U_{v,t}^{-}(\mathfrak{g})$-linear.

For $l\in \mathbb{Z}_{\geq 0},$ we set $Q_{-}(l)=\{\xi\in Q_{-}; |\xi|\leq l\}.$ For $\lambda,\mu\in P_{+}$ and $\xi\in Q_{-}(l)$, we will prove that the following statements are true, which would complete the proof of Theorem 5.1. We will modify Kashiwara's grand loop argument, and will only show the parts which differ most from [Kas, $\S4$].

$\mathbf{A}_{l}$ : $\tilde{e}_{i}L(\infty)_{\xi}\subset L(\infty)$ for any $i\in I.$

$\mathbf{B}_{l}$ : $\tilde{e}_{i}L(\lambda)_{\lambda+\xi}\subset L(\lambda)$ for any $i\in I.$

$\mathbf{C}_{l}$ : $\pi_{\lambda}(L(\infty)_{\xi})=L(\lambda)_{\lambda+\xi}.$

$\mathbf{D}_{l}$ : $B(\infty)_{\xi}$ is a $\mathbb{Q}(t)$-basis of $L(\infty)_{\xi}/vL(\infty)_{\xi}.$

$\mathbf{E}_{l}$ : $B(\lambda)_{\lambda+\xi}$ is a $\mathbb{Q}(t)$-basis of $L(\lambda)_{\lambda+\xi}/vL(\lambda)_{\lambda+\xi}.$

$\mathbf{F}_{l}$ : For any $\xi\in Q_{-}(l-1)$ and $P\in L(\infty)_{\xi},$ we have $\tilde{f}_{i}(Py_{\lambda})\equiv (\tilde{f}_{i}P)y_{\lambda}~\mathrm{mod}~vL(\lambda).$

$\mathbf{G}_{l}$ : $\tilde{e}_{i}B(\infty)_{\xi}\subset B(\infty)\cup\{0\}$ and $\tilde{e}_{i}B(\lambda)_{\lambda+\xi}\subset B(\lambda)\cup\{0\}$ for any $i\in I.$

$\mathbf{H}_{l}$ : $\Phi_{\lambda, \mu}(L(\lambda+\mu)_{\lambda+\mu+\xi})\subset L(\lambda)\otimes L(\mu).$

$\mathbf{I}_{l}$ : $\Psi_{\lambda, \mu}((L(\lambda)\otimes L(\mu))_{\lambda+\mu+\xi})\subset L(\lambda+\mu).$

$\mathbf{J}_{l}$ : $\Psi_{\lambda, \mu}((B(\lambda)\otimes B(\mu))_{\lambda+\mu+\xi})\subset B(\lambda+\mu)\cup\{0\}.$

$\mathbf{K}_{l}$ : We have $\{b\in B(\infty)_{\xi}; \bar{\pi}_{\lambda}(b)\neq 0\}\xrightarrow{\sim} B(\lambda)_{\lambda+\xi}.$

$\mathbf{L}_{l}$ : If $b\in B(\infty)_{\xi}$ satisfies $\bar{\pi}_{\lambda}(b)\neq 0,$ then we have $\tilde{e}_{i}\bar{\pi}_{\lambda}(b)=\bar{\pi}_{\lambda}(\tilde{e}_{i}b).$

$\mathbf{M}_{l}$ : For any $b\in B(\lambda)_{\lambda+\xi}$ and $b'\in B(\lambda)_{\lambda+\xi+\alpha_{i}},$ we have $b=\tilde{f}_{i}b'$ if and only if $b'=\tilde{e}_{i}b.$

$\mathbf{N}_{l}$ : For any $i$ and $b\in B(\infty)_{\xi}$ such that $\tilde{e}_{i}b\neq 0,$ we have $b=\tilde{f}_{i}\tilde{e}_{i}b.$

Note that the above statements are true for $l=0$ and $l=1.$ From now on, we assume $l\geq 2$ and will prove $\mathbf{A}_{l},\ldots,\mathbf{N}_{l},$ assuming that $\mathbf{A}_{l-1},\ldots,\mathbf{N}_{l-1}$ are true.

\begin{lemma}{\rm (See [Kas2, Lemma 4.3.1].)}
Let $\xi\in Q_{-}(l-1)$, $\lambda \in P_{+},$ and $y\in L(\infty)_{\xi}$ $($resp. $L(\lambda)_{\lambda+\xi})$. If $y=\sum_{i}f_{i}^{(n)}y_{n}$ and if $e_{i}'y_{n}=0$ $($resp. $y_{n}\in \mathrm{Ker}~e_{i}\cap V(\lambda)_{\lambda+\xi+n\alpha_{i}},$ and $y_{n}=0$ except when $\langle h_{i}, \lambda+\xi+n\alpha_{i}\rangle\geq n\geq 0)$, then all $y_{n}$ belong to $L(\infty)$ $($resp. $L(\lambda)$). Moreover, if $y ~\mathrm{mod}$ $vL(\infty)$ $($resp. $vL(\lambda))$ belongs to $B(\infty)$ $($resp. $B(\lambda))$, then there exists $n_{0}$ such that $y\equiv f_{i}^{(n_{0})}y_{n_{0}}~\mathrm{mod}~ vL(\infty)$ $($resp. $vL(\lambda))$.
\end{lemma}
\begin{proof}
By $\mathbf{A}_{l-1}$, we have $\tilde{e}_{i}^{r}y\in L(\infty)$ for all $r.$ Let $m$ be the largest integer such that $y_{m}\notin L(\infty).$ Then $\tilde{e}_{i}^{m}y=\sum_{n\geq m}f_{i}^{(n-m)}y_{n}$. Since $y_{n}\in L(\infty)$ for $n> m,$ thus we have $y_{m}\in L(\infty),$ which contradicts the choice of $m.$ Therefore, $y_{n}\in L(\infty)$ for all $n.$ A similar proof applies to $L(\lambda).$

Now suppose $y+vL(\infty)\in B(\infty).$ Let $n_{0}$ be the largest integer such that $y_{n_{0}}\notin vL(\infty).$ Then $\tilde{e}_{i}^{n_{0}}y+vL(\infty)=y_{n_{0}}+vL(\infty)\neq vL(\infty).$ By $\mathbf{N}_{k}$ for $k\leq l-1,$ $y+vL(\infty)=\tilde{f}_{i}^{n_{0}}\tilde{e}_{i}^{n_{0}}y+vL(\infty)=f_{i}^{(n_{0})}y_{n_{0}}+vL(\infty).$ The case of $L(\lambda)$ can be proved similarly.
\end{proof}

The following is a two-parameter analogue of [Kas2, Lemma 4.3.2].
\begin{lemma}Suppose $\xi, \xi'\in Q_{-}(l-1),$ $\lambda,\mu \in P_{+},$ $i\in I,$ and $b\in B(\lambda)_{\lambda+\xi},$ $b'\in B(\mu)_{\mu+\xi'}.$

$(\mathrm{i})$ $\tilde{f}_{i}(L(\lambda)_{\lambda+\xi}\otimes L(\mu)_{\mu+\xi'})\subset L(\lambda)\otimes L(\mu)$ and $\tilde{e}_{i}(L(\lambda)_{\lambda+\xi}\otimes L(\mu)_{\mu+\xi'})\subset L(\lambda)\otimes L(\mu).$

$(\mathrm{ii})$ We have the tensor product rule in $(L(\lambda)\otimes L(\mu))/v(L(\lambda)\otimes L(\mu))$ as follows:
\[\tilde{f}_{i}(b\otimes b')=\begin{cases}\tilde{f}_{i}b\otimes b'& \mathrm{if}~ \varphi_{i}(b)> \varepsilon_{i}(b'), \\ b\otimes\tilde{f}_{i}b'& \mathrm{if}~
\varphi_{i}(b)\leq \varepsilon_{i}(b');
\end{cases}\]
\[\tilde{e}_{i}(b\otimes b')=\begin{cases}\tilde{e}_{i}b\otimes b'& \mathrm{if}~ \varphi_{i}(b)\geq \varepsilon_{i}(b'), \\ b\otimes\tilde{e}_{i}b'& \mathrm{if}~
\varphi_{i}(b)< \varepsilon_{i}(b').
\end{cases}\]

$(\mathrm{iii})$ If $\tilde{e}_{i}(b\otimes b')\neq 0,$ then $b\otimes b'=\tilde{f}_{i}\tilde{e}_{i}(b\otimes b').$

$(\mathrm{iv})$ If $\tilde{e}_{i}(b\otimes b')=0$ for any $i\in I,$ then $\xi=0$ and $b=y_{\lambda}+vL(\lambda).$

$(\mathrm{v})$ We have $\tilde{f}_{i}(b\otimes y_{\mu})=\tilde{f}_{i}b\otimes y_{\mu}$ or $\tilde{f}_{i}b=0.$

$(\mathrm{vi})$ For any sequences of indices $i_{1},\ldots,i_{l},$ we have $$\tilde{f}_{i_{1}}\cdots \tilde{f}_{i_{l}}y_{\lambda}\in vL(\lambda) ~~~\mathrm{or}~~~\tilde{f}_{i_{1}}\cdots \tilde{f}_{i_{l}}(y_{\lambda}\otimes y_{\mu})\equiv \tilde{f}_{i_{1}}\cdots \tilde{f}_{i_{l}}y_{\lambda}\otimes y_{\mu}~\mathrm{mod}~vL(\lambda)\otimes L(\mu).$$
\end{lemma}
\begin{proof}
$(\mathrm{i})$ By Lemma 5.2, it is enough to prove the following statement: for $y\in \mathrm{Ker}~e_{i}\cap L(\lambda)_{\lambda+\xi+n\alpha_{i}}$ and $z\in \mathrm{Ker}~e_{i}\cap L(\mu)_{\mu+\xi'+m\alpha_{i}}$ with $\langle h_{i}, \lambda+\xi+n\alpha_{i}\rangle\geq n\geq 0$ and $\langle h_{i}, \mu+\xi'+m\alpha_{i}\rangle\geq m\geq 0,$ we have $$\tilde{f}_{i}(f_{i}^{(n)}y\otimes f_{i}^{(m)}z)\subset L(\lambda)\otimes L(\mu),~~~~\tilde{e}_{i}(f_{i}^{(n)}y\otimes f_{i}^{(m)}z)\subset L(\lambda)\otimes L(\mu).$$
Let $L$ be the free $\mathbf{A}$-submodule generated by $f_{i}^{(s)}y\otimes f_{i}^{(t)}z$ $(s,t\geq 0)$, then by the tensor product rule for $U_{v,t}(\mathfrak{g_{i}})$-modules, we have $\tilde{e}_{i}L\subset L$ and $\tilde{f}_{i}L\subset L.$ Our assertion follows immediately from the fact that $L\subset L(\lambda)\otimes L(\mu).$

$(\mathrm{ii})$-$(\mathrm{vi})$. These follow immediately from $(\mathrm{i})$, Lemma 5.2, and Theorem 3.4.
\end{proof}
\begin{lemma}{\rm (Compare [Kas2, Lemma 4.3.9].)}
Let $\xi\in Q_{-}(l)$ and $y\in V(\lambda)_{\lambda+\xi},$ and $n,k\in \mathbb{Z}_{\geq 0}$ with $n+k\geq 1.$ Assume $k_{i}'^{-\nu}e_{i}^{(\nu)}y\in v_{i}^{\nu(\nu+n+k)}vL(\lambda)$ for any $1\leq \nu\leq n+k.$ Then we have $$\tilde{f}_{i}^{n}f_{i}^{(k)}y\equiv f_{i}^{(k+n)}y~\mathrm{mod}~vL(\lambda),~~~~\tilde{e}_{i}^{n}f_{i}^{(k)}y\equiv f_{i}^{(k-n)}y~\mathrm{mod}~vL(\lambda).$$
\end{lemma}
\begin{proof}
Let $y=\sum f_{i}^{(m)}y_{m}$ be the $i$-string decomposition of $y.$ Then we have, setting $a=\langle h_{i}, \lambda+\xi\rangle,$
\begin{align*}
k_{i}'^{-\nu}e_{i}^{(\nu)}f_{i}^{(m)}y_{m}&=k_{i}'^{-\nu}f_{i}^{(m-\nu)}\left[\hspace*{-0.1cm}\begin{array}{c}\nu-m+\langle h_{i}, \lambda+\xi+m\alpha_{i}\rangle \\
\nu
\end{array}\hspace*{-0.1cm}\right]_{v_{i}}t^{\nu(\langle i, \lambda+\xi\rangle-\langle \lambda+\xi, i\rangle)}y_{m}\\
&=v_{i}^{\nu(a+2\nu)}\left[\hspace*{-0.1cm}\begin{array}{c}\nu+a+m \\
\nu
\end{array}\hspace*{-0.1cm}\right]_{v_{i}}f_{i}^{(m-\nu)}y_{m}.
\end{align*}
By Lemma 5.1, we obtain $v_{i}^{\nu(a+2\nu)}v_{i}^{-\nu(a+m)}y_{m}\in v_{i}^{\nu(\nu+n+k)}vL(\lambda)$ for $m\geq \nu$ and $1\leq \nu\leq n+k.$ Hence we obtain $v_{i}^{\nu(\nu-n-m-k)}y_{m}\in vL(\lambda)$ for $m\geq \nu$ and $1\leq \nu\leq n+k.$ Hence setting $\nu=n+k$ when $m\geq n+k$ and $\nu=m$ when $0< m< n+k,$ we obtain $v_{i}^{-m(n+k)}y_{m}\in vL(\lambda)$ for $m> 0.$

Now we have $$\tilde{f}_{i}^{n}f_{i}^{(k)}y=\sum \left[\hspace*{-0.1cm}\begin{array}{c}m+k \\
m
\end{array}\hspace*{-0.1cm}\right]_{v_{i}}f_{i}^{(n+k+m)}y_{m},~~~f_{i}^{(n+k)}y=\sum \left[\hspace*{-0.1cm}\begin{array}{c}m+n+k \\
m
\end{array}\hspace*{-0.1cm}\right]_{v_{i}}f_{i}^{(n+k+m)}y_{m}.$$
Therefore, both $\tilde{f}_{i}^{n}f_{i}^{(k)}y$ and $f_{i}^{(n+k)}y$ are equal to $f_{i}^{(n+k)}y_{0}$ modulo $vL(\lambda).$ The second equality can be proved similarly.
\end{proof}

In the sequel, for $\lambda\in P_{+},$ $`\lambda\gg 0$' means $`\langle h_{i}, \lambda\rangle\gg 0$' for all $i\in I.$
\begin{lemma}
Let $\xi\in Q_{-}(l)$ and $P\in U_{v,t}^{-}(\mathfrak{g})_{\xi}.$ Then for $\lambda\gg 0$, we have $$(\tilde{f}_{i}P)y_{\lambda}\equiv\tilde{f}_{i}(Py_{\lambda})~\mathrm{mod}~vL(\lambda)~~~\mathrm{and}~~~(\tilde{e}_{i}P)y_{\lambda}\equiv\tilde{e}_{i}(Py_{\lambda})~\mathrm{mod}~vL(\lambda).$$
\end{lemma}
\begin{proof}
We may assume $P=f_{i}^{(k)}R$ with $e_{i}'R=0$ and $R\in U_{v,t}^{-}(\mathfrak{g})_{\xi+k\alpha_{i}}.$ Then $(\tilde{f}_{i}P)y_{\lambda}=f_{i}^{(k+1)}Ry_{\lambda}$ and $(\tilde{e}_{i}P)y_{\lambda}=f_{i}^{(k-1)}Ry_{\lambda}.$ By Corollary 4.4, we have $$k_{i}'^{-\nu}e_{i}^{(\nu)}Ry_{\lambda}\in v_{i}^{\nu(\nu+k+1)}vL(\lambda)~~~\mathrm{for}~1\leq \nu \leq 1+k.$$ Then the lemma follows from Lemma 5.4.
\end{proof}

Let us denote by $L(\lambda)^{*}$ and $L(\infty)^{*}$ the dual lattice of $L(\lambda)$ and $L(\infty)$ with respect to the inner product defined in Proposition 3.6 and 4.8, respectively. This means $$L(\lambda)^{*}=\{y\in V(\lambda); (y, L(\lambda))\subset \mathbf{A}\},~~~~L(\infty)^{*}=\{y\in U_{v,t}^{-}(\mathfrak{g}); (y, L(\infty))\subset \mathbf{A}\}.$$ Similarly, we can define $L(\lambda)_{\lambda+\xi}^{*}$ and $L(\infty)_{\xi}^{*}.$
\begin{lemma}
For $\xi=-\sum n_{i}\alpha_{i}\in Q_{-}$ and $P, Q\in U_{v,t}^{-}(\mathfrak{g})_{\xi},$ there exists a polynomial $f(p_{1},\ldots,p_{n})$ in $p=(p_{i})_{i\in I}$ with coefficients in $\mathbb{Q}(v, t)$ such that $$(Py_{\lambda}, Qy_{\lambda})=f(p)~~~\mathrm{with}~p_{i}=v_{i}^{2\langle h_{i}, \lambda\rangle},~~\mathrm{and}$$$$f(0)=\Big(\prod(1-v_{i}^{2})^{-n_{i}}\Big)(P,Q).$$
\end{lemma}
\begin{proof}
We shall prove it by the induction on $|\xi|.$ If $|\xi|=0,$ it is obvious. When $|\xi|> 0,$ we may assume $Q=f_{i}R$ with $R\in U_{v,t}^{-}(\mathfrak{g})_{\xi+\alpha_{i}}.$
\begin{align*}
(Py_{\lambda}, Qy_{\lambda})&=v_{i}^{-1}(k_{i}'^{-1}e_{i}Py_{\lambda}, Ry_{\lambda})\\
&=v_{i}^{-1}\Big(\frac{k_{i}'^{-1}k_{i}e_{i}''(P)-e_{i}'(P)}{v_{i}-v_{i}^{-1}}y_{\lambda}, Ry_{\lambda}\Big)\\
&=(1-v_{i}^{2})^{-1}(e_{i}'(P)y_{\lambda}, Ry_{\lambda})-v_{i}^{2\langle h_{i}, \lambda+\xi+\alpha_{i}\rangle}(1-v_{i}^{2})^{-1}(e_{i}''(P)y_{\lambda}, Ry_{\lambda}).
\end{align*}
Hence the first equality follows. The second equality follows from $(P,Q)=(e_{i}'(P), R).$
\end{proof}

\begin{lemma}
For $\lambda\gg 0,$ we have $\pi_{\lambda}(L(\infty)_{\xi}^{*})=L(\lambda)_{\lambda+\xi}^{*}$ for any $\xi\in Q_{-}(l).$
\end{lemma}
\begin{proof}
Let $\{P_{k}\}$ be an $\mathbf{A}$-basis of $L(\infty)_{\xi}$ and $\{Q_{k}\}$ be the dual basis such that $(P_{k}, Q_{j})=\delta_{kj}.$ Then we have $L(\infty)_{\xi}^{*}=\sum_{j} \mathbf{A}Q_{j}.$ We also have $L(\lambda)_{\lambda+\xi}=\sum_{k}  \mathbf{A}P_{k}y_{\lambda}.$ By Lemma 5.6, we have $(P_{k}y_{\lambda}, Q_{j}y_{\lambda})\equiv \delta_{kj}~\mathrm{mod}~v\mathbf{A}$ for the choice of $\lambda.$ Hence we conclude $L(\lambda)_{\lambda+\xi}^{*}=\sum_{j} \mathbf{A}Q_{j}y_{\lambda}=\pi_{\lambda}(L(\infty)_{\xi}^{*}).$
\end{proof}

Using Lemma 5.7, we can also prove the following proposition by an argument similar to the proof of [Kas2, Proposition 4.7.3].
\begin{proposition}
Let $\mu\gg 0$ and $\xi\in Q_{-}(l)$. If $\lambda\in P_{+},$ then we have $$\Psi_{\lambda, \mu}((L(\lambda)\otimes L(\mu))_{\lambda+\mu+\xi})\subset L(\lambda+\mu)_{\lambda+\mu+\xi}.$$
\end{proposition}

The other parts of Kashiwara's grand loop argument work equally well in our two-parameter setting. Thus, Kashiwara's grand loop argument in [Kas2, $\S4$] together with the above modifications gets through, and we have established $\mathbf{A}_{l},\ldots,\mathbf{N}_{l}$ and completed the proof of Theorem 5.1.

\section{Properties of polarization}
Recall that we have introduced the polarizations on $V(\lambda)$ and $U_{v,t}^{-}(\mathfrak{g})$ respectively; see Proposition 3.6 and 4.8. In this section we shall investigate the properties of crystal bases with respect to the polarizations.

\begin{proposition}
Let $\lambda\in P_{+}.$

$(\mathrm{i})$ $(L(\lambda), L(\lambda))\subset \mathbf{A},$ and so it descends to a bilinear form $$(\cdot, \cdot)_{0} : L(\lambda)/vL(\lambda)\times L(\lambda)/vL(\lambda)\rightarrow \mathbb{Q}(t),~~(x+vL(\lambda), y+vL(\lambda))_{0}=(x,y)|_{v=0}.$$

$(\mathrm{ii})$ $(\tilde{e}_{i}x, y)_{0}=(x, \tilde{f}_{i}y)_{0}$ for $x,y\in L(\lambda)/vL(\lambda).$

$(\mathrm{iii})$ $B(\lambda)$ is an orthonormal basis with respect to $(\cdot, \cdot)_{0}.$ In particular, $(\cdot, \cdot)_{0}$ is positive definite.

$(\mathrm{iv})$ $L(\lambda)=\{x\in V(\lambda); (x, L(\lambda))\subset \mathbf{A}\}.$
\end{proposition}
\begin{proof}
We shall prove $(L(\lambda)_{\lambda+\xi}, L(\lambda)_{\lambda+\xi}) \subset \mathbf{A}$ by the induction on $|\xi|.$ If $|\xi|=0,$ it is trivial. Assume $|\xi|> 0.$ Since $L(\lambda)_{\lambda+\xi}=\sum \tilde{f}_{i}L(\lambda)_{\lambda+\xi+\alpha_{i}},$ it suffices to show $$(\tilde{f}_{i}x, y)\equiv (x, \tilde{e}_{i}y)~\mathrm{mod}~v\mathbf{A}~~\mathrm{for}~x\in L(\lambda)_{\lambda+\xi+\alpha_{i}}~\mathrm{and}~y\in L(\lambda)_{\lambda+\xi}.\eqno{(6.1)}$$
We may assume $x=f_{i}^{(n)}x_{0}$ and $y=f_{i}^{(m)}y_{0}$ with $e_{i}x_{0}=e_{i}y_{0}=0$, $\langle h_{i}, \lambda+\xi+(n+1)\alpha_{i}\rangle\geq n\geq 0$ and $\langle h_{i}, \lambda+\xi+m\alpha_{i}\rangle\geq m\geq 0.$

Then we have, setting $\mu=\lambda+\xi,$
\begin{align*}
(f_{i}^{(n+1)}x_{0}, f_{i}^{(m)}y_{0})&=\frac{1}{[m]_{v_{i}}^{!}}((v_{i}^{-1}k_{i}'^{-1}e_{i})^{m}f_{i}^{(n+1)}x_{0}, y_{0})\\
&=v_{i}^{-m^{2}}(k_{i}'^{-m}e_{i}^{(m)}f_{i}^{(n+1)}x_{0}, y_{0})\\
&=\delta_{n+1,m}v_{i}^{-m^{2}}(k_{i}'^{-m}\left[\hspace*{-0.1cm}\begin{array}{c}\langle h_{i}, \mu+(n+1)\alpha_{i}\rangle \\
m
\end{array}\hspace*{-0.1cm}\right]_{v_{i}}t^{m(\langle i, \mu\rangle-\langle \mu, i\rangle)}x_{0}, y_{0})\\
&=\delta_{n+1,m}v_{i}^{m(\langle h_{i}, \mu\rangle+m)}\left[\hspace*{-0.1cm}\begin{array}{c}\langle h_{i}, \mu\rangle+2m \\
m
\end{array}\hspace*{-0.1cm}\right]_{v_{i}}(x_{0}, y_{0})
\end{align*}
Since $(x_{0}, y_{0})\in \mathbf{A}$ by the hypothesis of induction and $v_{i}^{m(\langle h_{i}, \mu\rangle+m)}\left[\hspace*{-0.1cm}\begin{array}{c}\langle h_{i}, \mu\rangle+2m \\
m
\end{array}\hspace*{-0.1cm}\right]_{v_{i}}\in 1+v\mathbf{A},$ we obtain $$(f_{i}^{(n+1)}x_{0}, f_{i}^{(m)}y_{0})\equiv \delta_{n+1,m}(x_{0}, y_{0})~\mathrm{mod}~v\mathbf{A}.$$
Similar arguments show that $(f_{i}^{(n)}x_{0}, f_{i}^{(m-1)}y_{0})\equiv \delta_{n+1,m}(x_{0}, y_{0})~\mathrm{mod}~v\mathbf{A}.$ Hence we obtain (6.1), and whence $(\mathrm{i})$ and $(\mathrm{ii}).$

$(\mathrm{iii})$ We shall show that $(b,b')_{0}=\delta_{b,b'}$ for $b,b'\in B(\lambda)_{\lambda+\xi}$ by the induction on $|\xi|.$ If $|\xi|=0,$ this is trivial. If $|\xi|> 0,$ taking $i$ such that $\tilde{e}_{i}b\in B(\lambda),$ we have $(b,b')_{0}=(\tilde{f}_{i}\tilde{e}_{i}b,b')_{0}=(\tilde{e}_{i}b, \tilde{e}_{i}b')_{0}=\delta_{\tilde{e}_{i}b,\tilde{e}_{i}b'}=\delta_{b,b'}.$

$(\mathrm{iv})$ By $(\mathrm{i}),$ we obtain $L(\lambda)\subset\{x\in V(\lambda); (x, L(\lambda))\subset \mathbf{A}\}.$ Assume $x\in V(\lambda)$ such that $(x, L(\lambda))\subset \mathbf{A}$. By the definition of crystal basis, $x$ can be written as $x=\sum_{b\in B(\lambda)}a_{b}\tilde{b}$, where $a_{b}\in \mathbb{Q}(v, t)$ and $\tilde{b}+vL(\lambda)=b.$ Choose $r\in \mathbb{Z}_{\geq 0}$ be the smallest such that $v^{r}a_{b}\in \mathbf{A}$ for all $b.$ Assume $r> 0.$ Then $(v^{r}x, \tilde{b})=v^{r}(x, \tilde{b})\in v\mathbf{A}$ for all $b\in B(\lambda).$ So we have $$0=(\overline{v^{r}x}, b)_{0}=\sum(\overline{v^{r}a_{b'}}b', b)_{0}=v^{r}a_{b}+v\mathbf{A}.$$ Thus we have $v^{r-1}a_{b}\in \mathbf{A}$ for all $b\in B(\lambda).$ This contradicts the minimality of $r.$ Therefore we get $a_{b}\in \mathbf{A}$ for all $b\in B(\lambda)$ and $x\in L(\lambda).$
\end{proof}
Similar arguments show the following proposition.

\begin{proposition}

$(\mathrm{i})$ $(L(\infty), L(\infty))\subset \mathbf{A},$ and so it descends to a bilinear form $$(\cdot, \cdot)_{0} : L(\infty)/vL(\infty)\times L(\infty)/vL(\infty)\rightarrow \mathbb{Q}(t),~~(x+vL(\infty), y+vL(\infty))_{0}=(x,y)|_{v=0}.$$

$(\mathrm{ii})$ $(\tilde{e}_{i}x, y)_{0}=(x, \tilde{f}_{i}y)_{0}$ for $x,y\in L(\infty)/vL(\infty).$

$(\mathrm{iii})$ $B(\infty)$ is an orthonormal basis with respect to $(\cdot, \cdot)_{0}.$ In particular, $(\cdot, \cdot)_{0}$ is positive definite.

$(\mathrm{iv})$ $L(\infty)=\{x\in U_{v,t}^{-}(\mathfrak{g}); (x, L(\infty))\subset \mathbf{A}\}.$
\end{proposition}

The following result is an easy consequence of the positive definiteness of $(\cdot, \cdot)_{0}.$
\begin{proposition}
For $\lambda\in P_{+},$ we have $$L(\lambda)=\{x\in V(\lambda); (x, x)\subset \mathbf{A}\}~~~~\mathrm{and}~~~~L(\infty)=\{x\in U_{v,t}^{-}(\mathfrak{g}); (x, x)\subset \mathbf{A}\}.$$
\end{proposition}

\begin{lemma}
$(\mathrm{i})$ For any $i, j,$ we have $(Ad(k_{i})e_{i}'')\circ e_{j}'=e_{j}'\circ Ad(k_{i})e_{i}''$ in $\mathrm{End}(U_{v,t}^{-}(\mathfrak{g})).$

$(\mathrm{ii})$ We have $(Pf_{i}, Q)=(P, Ad(k_{i})e_{i}''Q)$ for any $P, Q\in U_{v,t}^{-}(\mathfrak{g}).$

\end{lemma}
\begin{proof}
Part $(\mathrm{i})$ follows immediately from Proposition 4.2.

Let us prove $(\mathrm{ii})$. When $P=1,$ $(f_{i}, f_{i})=(1, Ad(k_{i})e_{i}''f_{i})$ implies that $(\mathrm{ii})$ is true for any $Q.$ Hence it suffices to show that, if $P$ satisfies $(\mathrm{ii})$ for any $Q,$ then we have $(f_{j}Pf_{i}, Q)=(f_{j}P, Ad(k_{i})e_{i}''Q).$ By using $(\mathrm{i})$, we have \begin{align*}(f_{j}Pf_{i}, Q)&=(Pf_{i}, e_{j}'Q)=(P, Ad(k_{i})e_{i}''e_{j}'Q)\\&=(P, e_{j}'(Ad(k_{i})e_{i}'')Q)\\&=(f_{j}P, Ad(k_{i})e_{i}''Q).\end{align*}
Thus we obtain the desired result.
\end{proof}
\begin{lemma}
We have $(e_{i}'(P^{*}))^{*}= Ad(k_{i})e_{i}''P$ for any $P\in U_{v,t}^{-}(\mathfrak{g}).$
\end{lemma}
\begin{proof}
We have
\begin{align*}
[e_{i}, P]&=\frac{k_{i}e_{i}''(P)-k_{i}'e_{i}'(P)}{v_{i}-v_{i}^{-1}}\\
&=\frac{(Ad(k_{i})e_{i}''P)k_{i}-(Ad(k_{i}')e_{i}'P)k_{i}'}{v_{i}-v_{i}^{-1}}.
\end{align*}
Hence taking $*,$ we obtain $$[P^{*}, e_{i}]=\frac{k_{i}'(Ad(k_{i})e_{i}''P)^{*}-k_{i}(Ad(k_{i}')e_{i}'P)^{*}}{v_{i}-v_{i}^{-1}}.$$
Thus we obtain the desired result.
\end{proof}
\begin{proposition}
For any $P, Q\in U_{v,t}^{-}(\mathfrak{g}),$ we have $(P^{*}, Q^{*})=(P, Q).$
\end{proposition}
\begin{proof}
Since the proposition is true for $P=1,$ it suffices to prove that $(P^{*}, Q^{*})=(P, Q)$ implies that $((Pf_{i})^{*}, Q^{*})=(Pf_{i}, Q).$

We have, by Lemma 6.4$(\mathrm{ii})$ and Lemma 6.5,  \begin{align*}
((Pf_{i})^{*}, Q^{*})&=(f_{i}(P)^{*}, Q^{*}))=(P^{*}, e_{i}'Q^{*})\\&=(P, (e_{i}'Q^{*})^{*})
=(P, Ad(k_{i})e_{i}''Q)\\&=(Pf_{i}, Q).
\end{align*}\end{proof}
From Proposition 6.3 and 6.6, we immediately get the following result.
\begin{proposition} $L(\infty)^{*}=L(\infty).$
\end{proposition}

\section{Global crystal bases}

\subsection{The integral form of $U_{v,t}(\mathfrak{g})$ and $V(\lambda)$}

Let $\mathcal{A}=\mathbb{Z}[v^{\pm 1}, t^{\pm 1}]$ and $\left[\hspace*{-0.1cm}\begin{array}{c}k_{i};a \\ n \end{array}\hspace*{-0.1cm}\right]
=\prod\limits_{h=1}^{n}\frac{k_{i}v_{i}^{a-h+1}-k_{i}'v_{i}^{-(a-h+1)}}{v_{i}^{h}-v_{i}^{-h}}$ for $a\in \mathbb{Z}.$ Let us denote by $U_{v,t}^{\mathbb{Z}}(\mathfrak{g})$ the $\mathcal{A}$-subalgebra of $U_{v,t}(\mathfrak{g})$ generated by $e_{i}^{(n)},$ $f_{i}^{(n)},$ and $k_{i}^{\pm 1}$, $k_{i}'^{\pm 1},$ $\left[\hspace*{-0.1cm}\begin{array}{c}k_{i};0 \\ n \end{array}\hspace*{-0.1cm}\right]$ for $i\in I$ and $n\in \mathbb{Z}_{\geq 0}.$ We set $U_{\mathbb{Z}}^{-}(\mathfrak{g})$ to be the $\mathcal{A}$-subalgebra generated by $f_{i}^{(n)}$ for $i\in I$ and $n\in \mathbb{Z}_{\geq 0}.$ Then $U_{v,t}^{\mathbb{Z}}(\mathfrak{g})$ and $U_{\mathbb{Z}}^{-}(\mathfrak{g})$ are stable under $*$ and $-.$ Moreover, $U_{\mathbb{Z}}^{-}(\mathfrak{g})$ is stable by $e_{i}'$ whence $U_{\mathbb{Z}}^{-}(\mathfrak{g})$ is stable by the Kashiwara operators $\tilde{e}_{i}$ and $\tilde{f}_{i}.$ Therefore, we get $$y=\sum f_{i}^{(n)}y_{n}\in U_{\mathbb{Z}}^{-}(\mathfrak{g})~\mathrm{and}~e_{i}'y=0\Longrightarrow y_{n}\in U_{\mathbb{Z}}^{-}(\mathfrak{g}).$$

Let $(f_{i}^{n}U_{v, t}^{-}(\mathfrak{g}))^{\mathbb{Z}}=f_{i}^{n}U_{v, t}^{-}(\mathfrak{g})\cap U_{\mathbb{Z}}^{-}(\mathfrak{g}).$ Then $(f_{i}^{n}U_{v, t}^{-}(\mathfrak{g}))^{\mathbb{Z}}=\sum_{k\geq n}f_{i}^{(k)}U_{\mathbb{Z}}^{-}(\mathfrak{g})$ for $n\geq 0.$ Moreover, $y=\sum f_{i}^{(k)}y_{k}\in (f_{i}^{n}U_{v, t}^{-}(\mathfrak{g}))^{\mathbb{Z}}$ if and only if $y_{k}=0$ for $k< n.$

Let $L_{\mathbb{Z}}(\infty)=L(\infty)\cap U_{\mathbb{Z}}^{-}(\mathfrak{g}).$ Then $L_{\mathbb{Z}}(\infty)$ is stable by $\tilde{e}_{i}$ and $\tilde{f}_{i}.$ Therefore, we have $B(\infty) \subset L_{\mathbb{Z}}(\infty)/vL_{\mathbb{Z}}(\infty) \subset L(\infty)/vL(\infty).$

Let $\mathbf{A}_{\mathbb{Z}}$ be the $\mathbb{Z}[t, t^{-1}]$-subalgebra of $\mathbb{Q}(v, t)$ generated by $v$ and $(1-v^{2n})^{-1}$ $(n\geq 1)$. Let $\mathbf{K}_{\mathbb{Z}}$ be the subalgebra generated by $\mathbf{A}_{\mathbb{Z}}$ and $v^{-1}.$ Then we have $\mathbf{A}_{\mathbb{Z}}=\mathbf{A}\cap \mathbf{K}_{\mathbb{Z}}.$ We can easily see that $(U_{\mathbb{Z}}^{-}(\mathfrak{g}), U_{\mathbb{Z}}^{-}(\mathfrak{g}))\subset \mathbf{K}_{\mathbb{Z}},$ and hence $(L_{\mathbb{Z}}(\infty), L_{\mathbb{Z}}(\infty))\subset \mathbf{A}_{\mathbb{Z}}.$ Thus we obtain that $(\cdot, \cdot)_{0}$ is a $\mathbb{Z}[t, t^{-1}]$-valued bilinear form on $L_{\mathbb{Z}}(\infty)/vL_{\mathbb{Z}}(\infty).$ From this we can easily get the following proposition.
\begin{proposition}
$(\mathrm{i})$ $L_{\mathbb{Z}}(\infty)/vL_{\mathbb{Z}}(\infty)$ is a free $\mathbb{Z}[t, t^{-1}]$-module with a basis $B(\infty)$.

$(\mathrm{ii})$ $B(\infty)\cup (-B(\infty))=\{x\in L_{\mathbb{Z}}(\infty)/vL_{\mathbb{Z}}(\infty); (x, x)_{0}=1\}.$

\end{proposition}
From Proposition 6.6, 6.7 and 7.1, we can get the following corollary.
\begin{corollary}
$L_{\mathbb{Z}}(\infty)^{*}=L_{\mathbb{Z}}(\infty)$ and $B(\infty)^{*}\subset B(\infty)\cup (-B(\infty)).$
\end{corollary}

In fact, we have the following theorem
\begin{theorem}
$B(\infty)^{*}= B(\infty)$
\end{theorem}
\begin{proof}
We first define the operators $\tilde{e}_{i}^{*}$ and $\tilde{f}_{i}^{*}$ of $U_{v, t}^{-}(\mathfrak{g})$ by $\tilde{e}_{i}^{*}=*\tilde{e}_{i}*$ and $\tilde{f}_{i}^{*}=*\tilde{f}_{i}*.$ Then $L(\infty)$ and $L_{\mathbb{Z}}(\infty)$ are stable by $\tilde{e}_{i}^{*}$ and $\tilde{f}_{i}^{*}.$ By Lemma 6.5, we have $*e_{i}'*=Ad(k_{i})e_{i}''.$ Hence for any $P=\sum P_{n}f_{i}^{(n)}$ with $e_{i}''P_{n}=0,$ we have $P^{*}=\sum f_{i}^{(n)}P_{n}^{*}$ with $e_{i}'(P_{n}^{*})=0,$ whence $\tilde{f}_{i}^{*} P=\sum P_{n}f_{i}^{(n+1)}.$

By the induction of weight, it is enough to show that $b\in B(\infty),$ $\tilde{e}_{i}^{*}b=0$ implies $\tilde{f}_{i}^{*m}b\in B(\infty)$ for $m\geq 0.$ Take a representative $P\in L(\infty)$ of $b$ with $e_{i}''P=0.$ By the above claim, we get $\tilde{f}_{i}^{*m}b\equiv Pf_{i}^{(m)}~\mathrm{mod}~vL(\infty).$

Choose $\lambda\in P_{+}$ such that $\langle h_{i}, \lambda\rangle=0$ and $\langle h_{j}, \lambda\rangle\gg 0$ for any $j\neq i$ and $\mu\in P_{+}$ such that $\langle h_{j}, \mu\rangle\gg 0$ for any $j.$ We have $f_{i}^{(m)}(y_{\lambda}\otimes y_{\mu})=y_{\lambda}\otimes f_{i}^{(m)}y_{\mu}.$ On the other hand, $\Delta(P)=P\otimes 1~\mathrm{mod}~\sum_{\xi\neq 0}U_{v,t}(\mathfrak{g})\otimes U_{v,t}^{-}(\mathfrak{g})_{\xi}$ implies $$Pf_{i}^{(m)}(y_{\lambda}\otimes y_{\mu})=Py_{\lambda}\otimes f_{i}^{(m)}y_{\mu}~\mathrm{mod}~\sum_{\xi\neq \mu-m\alpha_{i}}V(\lambda)\otimes V(\mu)_{\xi}.$$ By Proposition 7.1$(\mathrm{ii})$, $Pf_{i}^{(m)}+vL(\infty)\in B(\infty)\cup (-B(\infty)).$ Accordingly, $Pf_{i}^{(m)}(y_{\lambda}\otimes y_{\mu})+vL(\lambda)\otimes L(\mu)$ belongs to $B(\lambda)\otimes B(\mu)\cup \{0\}$ or $-(B(\lambda)\otimes B(\mu))\cup \{0\}.$ Since $Py_{\lambda}\otimes f_{i}^{(m)}y_{\mu}+vL(\lambda)\otimes L(\mu)\in B(\lambda)\otimes B(\mu),$ we have $Pf_{i}^{(m)}(y_{\lambda}\otimes y_{\mu})+vL(\lambda)\otimes L(\mu)\in B(\lambda)\otimes B(\mu),$ which implies $Pf_{i}^{(m)}+vL(\infty)\in B(\infty).$ We obtain the desired result.
\end{proof}

For each $\lambda\in P_{+},$ we set $V_{\mathbb{Z}}(\lambda)=U_{\mathbb{Z}}^{-}(\mathfrak{g})y_{\lambda}.$ Then $V_{\mathbb{Z}}(\lambda)$ is a $U_{v,t}^{\mathbb{Z}}(\mathfrak{g})$-module. We also set, for $n\geq 0,$ $$(f_{i}^{n}V(\lambda))^{\mathbb{Z}}=(f_{i}^{n}U_{v, t}^{-}(\mathfrak{g}))^{\mathbb{Z}}y_{\lambda}=\sum_{k\geq n}f_{i}^{(k)}V_{\mathbb{Z}}(\lambda).$$ Note that $V_{\mathbb{Z}}(\lambda)$ and $(f_{i}^{n}V(\lambda))^{\mathbb{Z}}$ are stable under $-,$ where $-$ is the involutive automorphism of $V(\lambda)$ defined by $\overline{Py_{\lambda}}=\overline{P}y_{\lambda}$ for $P\in U_{v, t}^{-}(\mathfrak{g}).$

Let $L_{\mathbb{Z}}(\lambda)=L(\lambda)\cap V_{\mathbb{Z}}(\lambda).$ Since $\pi_{\lambda}(L(\infty))=L(\lambda),$ we have $\pi_{\lambda}(L_{\mathbb{Z}}(\infty))\subset L_{\mathbb{Z}}(\lambda),$ and hence $B(\lambda) \subset L_{\mathbb{Z}}(\lambda)/vL_{\mathbb{Z}}(\lambda) \subset L(\lambda)/vL(\lambda).$ It is proved similarly as in Proposition 7.1 that $L_{\mathbb{Z}}(\lambda)/vL_{\mathbb{Z}}(\lambda)$ is a free $\mathbb{Z}[t, t^{-1}]$-module with a basis $B(\lambda)$.

\begin{proposition}
Let $M$ be an integrable $U_{v,t}(\mathfrak{g})$-module and let $M_{\mathbb{Z}}$ be a $U_{v,t}^{\mathbb{Z}}(\mathfrak{g})$-submodule of $M.$ Let $\lambda\in P_{+}$ and $i\in I.$ Suppose that $n=-\langle h_{i}, \lambda\rangle \geq 0.$ Then $$(M_{\mathbb{Z}})_{\lambda}=\sum_{k\geq n}f_{i}^{(k)}(M_{\mathbb{Z}})_{\lambda+k\alpha_{i}}.$$
\end{proposition}
This follows immediately from the following lemma.

\begin{lemma}
When $n\geq 1,$ we have $y=\sum_{k\geq n}(-1)^{k-n}v_{i}^{kn}\left[\hspace*{-0.1cm}\begin{array}{c}k-1 \\
k-n
\end{array}\hspace*{-0.1cm}\right]_{v_{i}}f_{i}^{(k)}e_{i}^{(k)}k_{i}'^{-k}y$ for any $y\in M_{\lambda}.$
\end{lemma}
\begin{proof}
We may assume $y=f_{i}^{(m)}z$ with $z\in \mathrm{Ker}~e_{i}\cap M_{\lambda+m\alpha_{i}}$ and $m\geq n.$ Since we have $v_{i}^{kn}k_{i}'^{-k}f_{i}^{(m)}z=v_{i}^{kn-2km}f_{i}^{(m)}k_{i}'^{-k}z=v_{i}^{kn-2km}f_{i}^{(m)}v^{ki\cdot(\lambda+m\alpha_{i})}t^{-k(\langle i, \lambda\rangle-\langle\lambda, i\rangle)}z=t^{-k(\langle i, \lambda\rangle-\langle\lambda, i\rangle)}f_{i}^{(m)}z,$ we obtain
\begin{align*}
\hspace*{-1.6cm}&\sum_{k\geq n}(-1)^{k-n}v_{i}^{kn}\left[\hspace*{-0.1cm}\begin{array}{c}k-1 \\
k-n
\end{array}\hspace*{-0.1cm}\right]_{v_{i}}f_{i}^{(k)}e_{i}^{(k)}k_{i}'^{-k}y\end{align*}\vskip-6mm
\begin{align*}
&=\sum_{k\geq n}(-1)^{k-n}\left[\hspace*{-0.1cm}\begin{array}{c}k-1 \\
k-n
\end{array}\hspace*{-0.1cm}\right]_{v_{i}}t^{-k(\langle i, \lambda\rangle-\langle\lambda, i\rangle)}f_{i}^{(k)}e_{i}^{(k)}f_{i}^{(m)}z\\
&=\sum_{k=n}^{m}(-1)^{k-n}\left[\hspace*{-0.1cm}\begin{array}{c}k-1 \\
k-n
\end{array}\hspace*{-0.1cm}\right]_{v_{i}}f_{i}^{(k)}\left[\hspace*{-0.1cm}\begin{array}{c}k-m+(2m-n) \\
k
\end{array}\hspace*{-0.1cm}\right]_{v_{i}}f_{i}^{(m-k)}z\\
&=\sum_{k=n}^{m}(-1)^{k-n}\left[\hspace*{-0.1cm}\begin{array}{c}k-1 \\
k-n
\end{array}\hspace*{-0.1cm}\right]_{v_{i}}\left[\hspace*{-0.1cm}\begin{array}{c}k+m-n \\
k
\end{array}\hspace*{-0.1cm}\right]_{v_{i}}\left[\hspace*{-0.1cm}\begin{array}{c}m \\
k
\end{array}\hspace*{-0.1cm}\right]_{v_{i}}f_{i}^{(m)}z\\
&=y.
\end{align*}
The last identity follows from [Kas2, (6.1.19)].
\end{proof}

\subsection{Existence of global crystal bases}

We first give two lemmas, whose proof is similar to that of [Kas2, Lemma 7.1.1 and 7.1.2]. Recall that $\mathbf{A}$ is the subring of $\mathbb{Q}(v,t)$ consisting of rational functions without poles at $v=0.$ Let $\bar{\mathbf{A}}$ be the subring of $\mathbb{Q}(v,t)$ consisting of rational functions without poles at $v=\infty.$

\begin{lemma}
Let $V$ be a finite-dimensional vector space over $\mathbb{Q}(v,t),$ $V_{\mathbb{Z}}$ an $\mathcal{A}$-submodule of $V,$ $L_{0}$ a free $\mathbf{A}$-submodule of $V,$ and $L_{\infty}$ a free $\bar{\mathbf{A}}$-submodule of $V$ such that $V\cong \mathbb{Q}(v,t)\otimes_{\mathcal{A}} V_{\mathbb{Z}}\cong\mathbb{Q}(v,t)\otimes_{\mathbf{A}} L_{0}\cong \mathbb{Q}(v,t)\otimes_{\bar{\mathbf{A}}} L_{\infty}.$

$(\mathrm{i})$ Suppose that $V_{\mathbb{Z}}\cap L_{0}\cap L_{\infty}\rightarrow (V_{\mathbb{Z}}\cap L_{0})/(V_{\mathbb{Z}}\cap vL_{0})$ is an isomorphism. Then $$V_{\mathbb{Z}}\cap L_{0}\cong \mathbb{Z}[v, t^{\pm 1}]\otimes_{\mathbb{Z}[t^{\pm 1}]}(V_{\mathbb{Z}}\cap L_{0}\cap L_{\infty}),$$
$$V_{\mathbb{Z}}\cap L_{\infty}\cong \mathbb{Z}[v^{-1}, t^{\pm 1}]\otimes_{\mathbb{Z}[t^{\pm 1}]}(V_{\mathbb{Z}}\cap L_{0}\cap L_{\infty}),$$
$$V_{\mathbb{Z}}\cong \mathbb{Z}[v^{\pm 1}, t^{\pm 1}]\otimes_{\mathbb{Z}[t^{\pm 1}]}(V_{\mathbb{Z}}\cap L_{0}\cap L_{\infty}),$$
$$V_{\mathbb{Z}}\cap L_{0}\cap L_{\infty}\xrightarrow{\sim} (V_{\mathbb{Z}}\cap L_{\infty})/(V_{\mathbb{Z}}\cap v^{-1}L_{\infty}),$$\vskip-10mm
\begin{align*}(\mathbb{Q}\otimes_{\mathbb{Z}} V_{\mathbb{Z}})\cap L_{0}\cap L_{\infty}&\cong \mathbb{Q}\otimes_{\mathbb{Z}} (V_{\mathbb{Z}}\cap L_{0}/V_{\mathbb{Z}}\cap vL_{0})\\&\cong (\mathbb{Q}(v)[t^{\pm 1}]\otimes_{\mathcal{A}} V_{\mathbb{Z}})\cap L_{0}/(\mathbb{Q}(v)[t^{\pm 1}]\otimes_{\mathcal{A}} V_{\mathbb{Z}})\cap vL_{0}.\end{align*}

$(\mathrm{ii})$ Let $E$ be a $\mathbb{Z}[t, t^{-1}]$-module and $\varphi: E\rightarrow V_{\mathbb{Z}}\cap L_{0}\cap L_{\infty}$ a $\mathbb{Z}[t, t^{-1}]$-linear homomorphism. Assume that\vskip1mm
$(\mathrm{a})$ $V_{\mathbb{Z}}=\mathcal{A}\varphi(E)$ and

$(\mathrm{b})$ $E\rightarrow L_{0}/vL_{0}$ and $E\rightarrow L_{\infty}/v^{-1}L_{\infty}$ are isomorphisms.\vskip1mm
\noindent Then we have $E\rightarrow V_{\mathbb{Z}}\cap L_{0}\cap L_{\infty}\rightarrow V_{\mathbb{Z}}\cap L_{0}/M\cap vL_{0}$ are isomorphisms.
\end{lemma}

\begin{lemma}
Let $V,$ $V_{\mathbb{Z}},$ $L_{0}$ and $L_{\infty}$ be as in the assumption of Lemma 7.6. Let $N$ be an $\mathcal{A}$-submodule of $V_{\mathbb{Z}}.$ Assume we have the following conditions$\mathrm{:}$\vskip1mm

$(\mathrm{i})$ $N\cap L_{0}\cap L_{\infty}\xrightarrow{\sim} (N\cap L_{0})/(N\cap vL_{0}).$

$(\mathrm{ii})$ There exist a $\mathbb{Z}[t, t^{-1}]$-module $F$ and a homomorphism $\psi: F\rightarrow V_{\mathbb{Z}}\cap (L_{0}+N)\cap (L_{\infty}+N)$ such that\vskip1mm

$(\mathrm{a})$ $V_{\mathbb{Z}}=\mathcal{A}\psi(F)+N$ and

$(\mathrm{b})$ the two homomorphisms induced by $\psi,$ $F\xrightarrow{\varphi} (L_{0}+\mathbb{Q}(v)[t^{\pm 1}]\otimes_{\mathcal{A}}N)/(vL_{0}+\mathbb{Q}(v)[t^{\pm 1}]\otimes_{\mathcal{A}}N)$ and $F\rightarrow (L_{\infty}+\mathbb{Q}(v)[t^{\pm 1}]\otimes_{\mathcal{A}}N)/(v^{-1}L_{\infty}+\mathbb{Q}(v)[t^{\pm 1}]\otimes_{\mathcal{A}}N)$ are injective.

Then we have \vskip1mm

$(\mathrm{i})$ $V_{\mathbb{Z}}\cap L_{0}\cap L_{\infty}\rightarrow V_{\mathbb{Z}}\cap L_{0}/V_{\mathbb{Z}}\cap vL_{0}$ is an isomorphism.

$(\mathrm{ii})$ $0\rightarrow N\cap L_{0}/N\cap vL_{0}\rightarrow V_{\mathbb{Z}}\cap L_{0}/V_{\mathbb{Z}}\cap vL_{0}\xrightarrow{\phi} (L_{0}+\mathbb{Q}(v)[t^{\pm 1}]\otimes_{\mathcal{A}}N)/(vL_{0}+\mathbb{Q}(v)[t^{\pm 1}]\otimes_{\mathcal{A}}N)$ is exact and $\phi(V_{\mathbb{Z}}\cap L_{0}/V_{\mathbb{Z}}\cap vL_{0})=\varphi(F).$
\end{lemma}

Let us consider the following collection $(G_{l})$ of statements for $l\geq 0.$

$(G_{l}.1)$ For any $\xi\in Q_{-}(l)$, $$U_{\mathbb{Z}}^{-}(\mathfrak{g})_{\xi}\cap L(\infty)\cap \overline{L(\infty)}\rightarrow L_{\mathbb{Z}}(\infty)_{\xi}/vL_{\mathbb{Z}}(\infty)_{\xi}$$
is an isomorphism.

$(G_{l}.2)$ For any $\xi\in Q_{-}(l)$ and $\lambda\in P_{+}$, $$V_{\mathbb{Z}}(\lambda)_{\lambda+\xi}\cap L(\lambda)\cap \overline{L(\lambda)}\rightarrow L_{\mathbb{Z}}(\lambda)_{\lambda+\xi}/vL_{\mathbb{Z}}(\lambda)_{\lambda+\xi}$$ is an isomorphism.

Let us denote by $b\mapsto G(b)$ and $b\mapsto G_{\lambda}(b)$ the inverses of these isomorphisms.

$(G_{l}.3)$ For any $\xi\in Q_{-}(l)$ and $n\geq 0$, $b\in \tilde{f}_{i}^{n}(B(\infty)_{\xi+n\alpha_{i}})$ implies that $G(b)\in f_{i}^{n}U_{v, t}^{-}(\mathfrak{g}).$

Note that when $l=0,$ these statements are obviously true. We shall prove $G_{l}$ by induction on $l.$ Let us assume $l> 0$ and $G_{l-1}$ holds.
\begin{lemma}
For any $\xi\in Q_{-}(l)$ and $\lambda\in P_{+}$, we have

$(\mathrm{i})$ $$U_{\mathbb{Z}}^{-}(\mathfrak{g})_{\xi}\cap L(\infty)=\bigoplus_{b\in B(\infty)_{\xi}}\mathbb{Z}[v, t^{\pm 1}]G(b),$$
$$U_{\mathbb{Z}}^{-}(\mathfrak{g})_{\xi}=\bigoplus_{b\in B(\infty)_{\xi}}\mathbb{Z}[v^{\pm 1}, t^{\pm 1}]G(b),$$
$$V_{\mathbb{Z}}(\lambda)_{\lambda+\xi}\cap L(\lambda)=\bigoplus_{b\in B(\lambda)_{\lambda+\xi}}\mathbb{Z}[v, t^{\pm 1}]G_{\lambda}(b),$$
$$V_{\mathbb{Z}}(\lambda)_{\lambda+\xi}=\bigoplus_{b\in B(\lambda)_{\lambda+\xi}}\mathbb{Z}[v^{\pm 1}, t^{\pm 1}]G_{\lambda}(b).$$

$(\mathrm{ii})$ $G(b)y_{\lambda}=G_{\lambda}(\bar{\pi}_{\lambda}(b))$ for any $b\in L_{\mathbb{Z}}(\infty)_{\xi}/vL_{\mathbb{Z}}(\infty)_{\xi}.$

$(\mathrm{iii})$ $\overline{G(b)}=G(b)$ for any $b\in L_{\mathbb{Z}}(\infty)_{\xi}/vL_{\mathbb{Z}}(\infty)_{\xi}$ and $\overline{G_{\lambda}(b)}=G_{\lambda}(b)$ for any $b\in L_{\mathbb{Z}}(\lambda)_{\lambda+\xi}/vL_{\mathbb{Z}}(\lambda)_{\lambda+\xi}.$
\end{lemma}
\begin{proposition}
For any $\xi\in Q_{-}(l)$, $\lambda\in P_{+}$, $n\geq 1$ and $i\in I$, we have
\begin{align*}(f_{i}^{n}V(\lambda))^{\mathbb{Z}}_{\lambda+\xi}\cap L(\lambda)\cap \overline{L(\lambda)}&\xrightarrow{\sim} (f_{i}^{n}V(\lambda))^{\mathbb{Z}}_{\lambda+\xi}\cap L(\lambda)/v((f_{i}^{n}V(\lambda))^{\mathbb{Z}}_{\lambda+\xi}\cap L(\lambda))\\&\simeq \bigoplus_{b\in B(\lambda)_{\lambda+\xi}\cap \tilde{f}_{i}^{n}B(\lambda)}\mathbb{Z}[t, t^{-1}]b,\end{align*}\vskip-6mm
\begin{align*}(f_{i}^{n}U_{v, t}^{-}(\mathfrak{g}))^{\mathbb{Z}}_{\xi}\cap L(\infty)\cap \overline{L(\infty)}&\xrightarrow{\sim} (f_{i}^{n}U_{v, t}^{-}(\mathfrak{g}))^{\mathbb{Z}}_{\xi}\cap L(\infty)/(f_{i}^{n}U_{v, t}^{-}(\mathfrak{g}))^{\mathbb{Z}}_{\xi}\cap vL(\infty)\\&\simeq \bigoplus_{b\in B(\infty)_{\xi}\cap \tilde{f}_{i}^{n}B(\infty)}\mathbb{Z}[t, t^{-1}]b.\end{align*}
\end{proposition}

The remaining components of the inductive proof of $(G_{l}.1)$-$(G_{l}.3)$ go forward just as in [Kas2, $\S7.4$-$\S7.5$].

We summarize the main results on global crystal bases as follows.
\begin{theorem}
$(\mathrm{i})$ $\{G(b); b\in B(\infty)\}$ is a bar-invariant $\mathcal{A}$-basis of $U_{\mathbb{Z}}^{-}(\mathfrak{g}),$ and a $\mathbb{Q}(v, t)$-basis of $U_{v, t}^{-}(\mathfrak{g}).$ Moreover, $G(b^{*})=G(b)^{*}$ for any $b\in B(\infty).$

$(\mathrm{ii})$ $\{G_{\lambda}(b); b\in B(\lambda)\}$ is a bar-invariant $\mathcal{A}$-basis of $V_{\mathbb{Z}}(\lambda)$, and a $\mathbb{Q}(v, t)$-basis of $V(\lambda).$ Moreover, $G(b)y_{\lambda}=G_{\lambda}(\bar{\pi}_{\lambda}(b))$ for any $b\in B(\infty)$ and $\lambda\in P_{+}.$

$(\mathrm{iii})$ For any $i\in I,$ $n\geq 0$ and $\lambda\in P_{+}$, we have
$$f_{i}^{n}U_{v, t}^{-}(\mathfrak{g})\cap U_{\mathbb{Z}}^{-}(\mathfrak{g})=\bigoplus_{b\in \tilde{f}_{i}^{n}B(\infty)}\mathcal{A}G(b),$$
$$f_{i}^{n}V(\lambda)\cap V_{\mathbb{Z}}(\lambda)=\bigoplus_{b\in B(\lambda)\cap \tilde{f}_{i}^{n}B(\lambda)}\mathcal{A}G_{\lambda}(b).$$
\end{theorem}

\noindent $\mathbf{Remark ~7.11.}$ From the constructions as above, we can see that the global crystal basis of $U_{v, t}^{-}(\mathfrak{g})$ over $\mathbb{Q}(v, t)$ coincides with that of $U_{v}^{-}(\mathfrak{g})$ over $\mathbb{Q}(v),$ so it is the same as the canonical basis constructed in [FL1] up to a 2-cocycle deformation.

\vskip3mm Acknowledgements. I am deeply indebted to Professor Weiqiang Wang and Dr. Zhaobing Fan for very helpful comments on a preliminary version.

%/////////////////////////////////////////////////////////////////////////////////////////////////////////////////////////////////////////////////////////

%by (\ref{stn})
%\begin{align*}
%  t_{n+1}&\geq\frac{1}{t_n+(1-t_n)\frac{1}{c}}\\
%  1-t_{n+1}&\leq \frac{(1-t_n)(\frac{1}{c}-1)}{t_n+(1-t_n)\frac{1}{c}}=\frac{(1-t_n)(\frac{1}{c}-1)}{(1-t_n)(\frac{1}{c}-1)+1}\\
%  \frac{1}{1-t_{n+1}}&\geq \frac{1}{(1-t_n)(\frac{1}{c}-1)+1}
%\end{align*}
%Since $c\geq \frac{1}{2}$, $\frac{1}{c}-1\geq 1$, we have
%$$\frac{1}{1-t_{n+1}}\geq \frac{1}{1-t_n}+1\geq\cdots\geq\frac{1}{1-t_1}+n\geq n+1$$
%i.e. $1-t_{n+1}\leq \frac{1}{n+1}$
%$$0\leq v_{2n+1}-x^*\leq v_{2n+1}-v_{2n}\leq v_{2n+1}-t_nv_{2n+1}\geq \frac{1}{n}v_{2n+1}$$
%So $\|v_{2n+1}-x^*\|\leq \frac{N}{n}\|\bar{v}\|$, where $N$ is the normal constant of $P$.
%/////////////////////////////////////////////////////////////////////////////////////////////////////////////////////////////////

Mathematical Sciences Center, Tsinghua University.

Beijing, 100084, P. R. China.

E-mail address: cwdeng@amss.ac.cn

\end{document}